\documentclass[11pt,a4paper]{amsart}

\usepackage[english]{babel} 
\usepackage[T1]{fontenc} 
\usepackage{amsmath}
\usepackage{amssymb} 
\usepackage{ucs} 
\usepackage[utf8x]{inputenc}
\usepackage[mathcal]{eucal} 
\usepackage[]{graphicx} 
\usepackage{psfrag}
\usepackage{latexsym}
\usepackage{amsthm} 
\usepackage{stmaryrd}
\usepackage{verbatim}
\usepackage{hyperref}

\usepackage{pgf,tikz}
\usetikzlibrary{arrows}
\usetikzlibrary[patterns]

\newtheorem{defi}{Definition}[section] 
\newtheorem{theo}[defi]{Theorem}
\newtheorem{coro}[defi]{Corollary} 
\newtheorem{lemme}[defi]{Lemma}
\newtheorem{lemma}[defi]{Lemma}
\newtheorem{prop}[defi]{Proposition}

\theoremstyle{remark}

\newcommand{\T}{\mathbb{T}} 
\newcommand{\Ss}{\mathbb{S}}

\newcommand{\R}{\mathbb{R}}

\newcommand{\Z}{\mathbb{Z}}
\newcommand{\proj}{\mathbb{P}}
\newcommand{\h}{\mathbb{H}}
\newcommand{\N}{\mathbb{N}} 

\newcommand{\e}{\varepsilon}
\newcommand{\p}{\varphi}
\newcommand{\F}{\mathcal{F}}
\newcommand{\Isom}{\mathrm{Isom}}
\newcommand{\Conf}{\mathrm{Conf}}
\newcommand{\Diff}{\mathrm{Diff}}
\newcommand{\Homeo}{\mathrm{Homeo}}

\newcommand{\suite}{_{n\in \N}}

\newcommand{\Stab}{\mathrm{Stab}}
\newcommand{\PSL}{\mathrm{PSL}}

\newcommand{\Log}{\mathrm{Log}}
\newcommand{\Aff}{\mathrm{Aff}}
\newcommand{\dS}{\mathrm{dS}}

\newcommand{\intoo}[2]{\mathopen{]}#1\,,#2\mathclose{[}}
\newcommand{\intff}[2]{\mathopen{[}#1\,,#2\mathclose{]}}
\newcommand{\intof}[2]{\mathopen{]}#1\,,#2\mathclose{]}}
\newcommand{\intfo}[2]{\mathopen{[}#1\,,#2\mathclose{[}}

\renewcommand{\tilde}{\widetilde}

\title{Isometries of Lorentz surfaces and convergence groups} 
\author{Daniel Monclair}
\date{\today}

\thanks{Partially supported by  ANR project GR-Analysis-Geometry (ANR-2011-BS01-003-02)}

\begin{document}

\maketitle

\begin{abstract} We study the isometry group of a globally hyperbolic  spatially compact Lorentz surface. Such a group acts on the circle, and we show that when the isometry group acts non properly,  the subgroups of $\mathrm{Diff}(\mathbb{S}^1)$ obtained  are semi conjugate to subgroups of finite covers of $\mathrm{PSL}(2,\mathbb{R})$ by  using  convergence groups. Under an assumption on the conformal boundary, we show that we have a conjugacy in $\mathrm{Homeo}(\mathbb{S}^1)$. \end{abstract}

\setcounter{tocdepth}{1}
\tableofcontents

\section{Introduction} One of the (many) differences between Riemannian and Lorentzian geometries lies in the dynamics of the isometry group. More precisely, the isometry group  of a Lorentz manifold can act non properly (stabilizers are subgroups of the non compact group $\mathrm O(n-1,1)$), which is impossible for Riemannian isometries ($\mathrm O(n)$ is compact). However, geometries with many isometries tend to be rare, and Gromov's vague conjecture (as stated in \cite{D'AG}) suggests that they should be classifiable. Several results classifying non proper actions on Lorentz manifolds exist, such as in the compact case (\cite{AS97a}, \cite{AS97b}, \cite{Z99a} and \cite{Z99b}) and for actions of  simple Lie groups on non compact manifolds (\cite{Kowalsky}). However, the study of Lorentz manifolds cannot be reduced to the  compact case. Indeed,  causality is an obstruction to compactness. It would be interesting to obtain a classification result for Lorentz manifolds with a large isometry group under a physically relevant assumption.\\
\indent Lorentz manifolds that occur in physics are globally hyperbolic. Global hyperbolicity is the strongest of all the causality conditions. It is roughly equivalent to the fact that the space of lightlike geodesics is Hausdorff (and therefore a smooth manifold).  We wish to classify globally hyperbolic space-times on which the isometry group acts non properly.\\
\indent We will work in the subcategory of spatially compact Lorentz manifolds, i.e. space-times such that the space of lightlike geodesics is a compact manifold. They are space-times for which space is compact.\\
\indent  While working with classification problems, it is standard to start in a  low dimensional setting. In Lorentzian geometry, the lowest dimension possible is $2$, so we will study spatially compact Lorentz surfaces in order to understand when the isometry group acts non properly.\\
\indent The first example of such a surface is the De Sitter space $\dS_2$, which is the one sheeted hyperboloïd $x^2+y^2-z^2=1$ in $\R^3$, endowed with the restriction of $dx^2+dy^2-dz^2$. The  isometry group   is $\mathrm{SO}^{\circ}(1,2) \approx \PSL(2,\R)$ (we only consider orientation and time orientation preserving isometries). It is diffeomorphic to $\Ss^1 \times \R$ (and so are all spatially compact surfaces), therefore it admits finite covers $\dS_2^k$ for all $k\in \N$,  whose isometry groups  are  the finite covers $\PSL_k(2,\R)$. We will see that spatially compact surfaces do not give more isometry groups:

\begin{theo} \label{algebraic} Let $(M,g)$ be a spatially compact surface such that $\Isom(M,g)$ acts non properly on $M$. The isometry group  $\Isom(M,g)$ is isomorphic to a subgroup of a finite cover of $\PSL(2,\R)$. \end{theo}

\subsection{One-dimensional dynamics}
This isomorphism arises from a one-dimensional structure. Given a Lorentz manifold $(M,g)$, the conformal group acts on the space of lightlike geodesics. For a spatially compact surface, it is diffeomorphic to two copies of the circle, which gives us two representations $\rho_1^M,\rho_2^M:\Conf(M,g)\to \Diff(\Ss^1)$. We will see that these representations tend to characterise the isometry group, and that they encode the dynamical properties of the isometry group. \\ 
\indent In the case of the De Sitter space $\dS_2$, the actions $\rho_1^{\dS_2}$ and $\rho_2^{\dS_2}$ are both equal to the projective action of  $\PSL(2,\R)$ on $\Ss^1\approx \R\proj^1$. For the finite covers $\dS_k^2$, the actions of   $\PSL_k(2,\R)$ are given by the groups whose elements are the lifts of elements of $\PSL(2,\R)$ to the $k$-covering of the circle.\\
\indent We follow \cite{Gh87} and say that two representations $\sigma,\tau : \Gamma \to \Homeo(\Ss^1)$ are \textbf{semi conjugate} if there is a non constant map $h:\Ss^1\to\Ss^1$ which is non decreasing  of degree one (i.e. that has a non decreasing lift $\tilde h:\R\to \R$ such that $\tilde h(x+1)=\tilde h(x) +1$) such that $h\circ \sigma(\gamma) = \tau(\gamma) \circ h$ for all $\gamma \in \Gamma$ (note that we do not ask for $h$ to be continuous).

\begin{theo} \label{main_theorem} Let $(M,g)$ be a spatially compact surface such that $\Isom(M,g)$ acts non properly on $M$. Then $\rho_1^M$ and $\rho_2^M$ are semi conjugate to each other, and the restrictions to $\Isom(M,g)$ are faithful. There are $k\in \N$ and a faithful representation $\rho : \Isom(M,g) \to \PSL_k(2,\R)$  that is semi conjugate to the restrictions of $\rho_1^M$ and $\rho_2^M$ to $\Isom(M,g)$. \end{theo}

This paper deals with two different theories: groups of circle diffeomorphisms and  (two dimensional) Lorentz manifolds. Some studies have already linked the two, concerning completeness of Lorentz metrics (\cite{Carriere_Rozoy}, with applications to Lorentzian geometry) and links between geodesic curvature and the Schwarzian derivative (see \cite{DO}). 

\subsection{Conformal boundary and semi-conjugacy}  One of the main tools throughout this paper is the use of global conformal models for spatially compact surfaces. Lorentz surfaces are locally conformally flat, but in the globally hyperbolic case, one can find two natural immersions in  flat spaces which are globally defined.\\
\indent The first one consists of a conformal immersion in the flat Lorentzian torus $p : (M,[g])\to (\Ss^1\times \Ss^1,[dxdy])$ (by \textbf{flat Lorentzian torus}, we will always mean $(\T^2,dxdy)$, which is conformal to the two dimensional Einstein Universe $\mathrm{Ein}^{1,1}$). It is defined for all spatially compact surfaces. In this model, an isometry $\p\in \Isom(M,g)$ acts on $p(M)$ by the diffeomorphism $(x,y)\mapsto (\rho_1^M(\p)(x),\rho_2^M(\p)(y))$ where $\rho_1^M,\rho_2^M$ are the representations defined above. In the case where $p$ is an embedding, the boundary of $p(M)$ consists in the graphs of non decreasing maps of degree one which provide a semi-conjugacy between $\rho_1^M$ and $\rho_2^M$. Furthermore, we do not need the non properness of the action to get the results of Theorem \ref{main_theorem}.

\begin{theo} \label{main_theorem_embeds} Let $(M,g)$ be a spatially compact surface that embeds conformally in the flat torus. Then $\rho_1^M$ and $\rho_2^M$ are semi conjugate to each other, and the restrictions to $\Isom(M,g)$ are faithful. There are $k\in \N$ and a faithful representation $\rho:\Isom(M,g) \to \PSL_k(2,\R)$ that is semi conjugate to the restrictions of $\rho_1^M$ and $\rho_2^M$ to $\Isom(M,g)$. \end{theo}

However, this conformal model will not be enough because $p$ is not always injective (it is injective if and only if there are no conjugate points along lightlike geodesics).\\
\indent The second conformal model consists in applying the first one to the universal cover, replacing the circle $\Ss^1$ by the real line $\R$. We get a conformal immersion $\tilde p : (\tilde M, [\tilde g]) \to (\R^2, [dxdy])$. This time, it is always an embedding, and it can be chosen in a way such that a generator of $\pi_1(M)\approx \Z$ acts on $\tilde p(\tilde M)$ by the map $(x,y)\mapsto (x+1,y+1)$, so $M$ is conformally diffeomorphic to an open set of the flat Lorentzian cylinder (quotient of the Minkowski plane by a spacelike translation). We will make use of this second model in order to see that, when the isometry group acts non properly, it is sufficient to study the case where the first conformal immersion $p:M\to \T^2$  is an embedding.\\
\indent This conformal model gives us a notion of conformal boundary (see \cite{FHS} for a general treatment of the conformal and causal boundaries of space-times), defined to be the boundary in the flat Lorentzian cylinder. We say that the conformal boundary of $(M,g)$ is \textbf{acausal} if the boundary in the flat Lorentzian cylinder does not contain any  segment of lightlike geodesic. Under this assumption, we obtain a topological conjugacy.

\begin{theo} \label{main_theorem_acausal} Let $(M,g)$ be a spatially compact surface such that $\Isom(M,g)$ acts non properly on $M$. Assume that the conformal boundary of $(M,g)$ is acausal. Then $\rho_1^M$ and $\rho_2^M$ are faithful, topologically conjugate to each other, and the restrictions to $\Isom(M,g)$ are topologically conjugate to a representation in a finite cover of $\PSL(2,\R)$. \end{theo}

The acausal character of the conformal boundary seems to be more than a technical tool, since counter examples exist without it, i.e. spatially compact surfaces $(M,g)$ such that the isometry group acts non properly, for which the restriction of $\rho_1^M$ to $\Isom(M,g)$ is not topologically conjugate to a subgroup of any $\PSL_k(2,\R)$ (see \cite{Monclairb} for a detailed construction).\\
\indent Once again, for open sets of $\T^2$, we do not need the non properness of the action of the isometry group.

\begin{theo} \label{main_theorem_acausal_embeds} Let $(M,g)$ be a spatially compact surface  that embeds conformally in the flat torus. Assume that the conformal boundary of $(M,g)$ is acausal. Then $\rho_1^M$ and $\rho_2^M$ are faithful, topologically conjugate to each other, and the restrictions to $\Isom(M,g)$ are topologically conjugate to a representation in a finite cover of $\PSL(2,\R)$.
\end{theo}

However, this does not mean that the semi conjugacy obtained in Theorem \ref{main_theorem} is a homeomorphism when the conformal boundary is acausal. Indeed, this is only true if the actions on the circle have dense orbits. 

\subsection{Convergence groups and generalisations}
Convergence groups play a major role in the study of circle homeomorphisms as they give a dynamical description of groups that are topologically conjugate to subgroups of $\PSL(2,\R)$.\\
\indent A subgroup $G\subset \Homeo(\Ss^1)$ is said to be a convergence group if all sequences in $G$ have  north\slash south dynamics. The first example is the group generated by a hyperbolic element of $\PSL(2,\R)$, but it is also the case for any subgroup of $\PSL(2,\R)$. Since the definition of a convergence group is stated in terms of topological dynamics, the notion is invariant under topological conjugacy, which means that all groups that are topologically conjugate to a subgroup of $\PSL(2,\R)$ are convergence groups. A Theorem of Gabai and Casson-Jungreis (\cite{Gabai}, \cite{CJ}) states that the converse is true: convergence groups are topologically conjugate to subgroups of $\PSL(2,\R)$.\\
\indent However, convergence groups will not be enough in order to describe Lorentz isometry groups as there are examples for which the groups of circle diffeomorphisms are not conjugate to subgroups of $\PSL(2,\R)$. Indeed, the actions of the groups  $\PSL_k(2,\R)$  do not satisfy the convergence property (lifts of hyperbolic elements have $2k$ fixed points). This will lead us to introduce a generalisation of convergence groups, by using more limit points:

\begin{defi} Let $k\in \N^*$ and let $h\in \Homeo(\Ss^1)$ have rotation number $\frac{1}{k}$. A sequence $(f_n)\suite\in\Homeo(\Ss^1)^\N$ has the $(h,k)$-convergence property if there are $a,b\in \Ss^1$ such that, up to a subsequence: \begin{itemize} \item $h^k(a)=a$ and $h^k(b)=b$ \item $\forall i\in\{0,\dots,k-1\}~ \forall x\in \intoo{h^i(a)}{h^{i+1}(a)} ~ f_n(x)\to h^i(b)$ \end{itemize} A group  $G\subset \Homeo(\Ss^1)$  is a $(h,k)$-convergence group if all elements of $G$ commute with $h$ and all sequence in $G$ either has the $(h,k)$-convergence property or has an equicontinuous subsequence.
\end{defi}

When $k=1$, then this definition is equivalent to the convergence property.  Subgroups of $\PSL_k(2,\R)$ have the $(h,k)$-convergence property, where $h$ is the rotation of angle $\frac{1}{k}$. Conversely, if  $G$ has the $(h,k)$-convergence property and $h$ is topologically conjugate to the rotation of angle $\frac{1}{k}$, then it is quite straightforward that $G$ is topologically conjugate to a subgroup of $\PSL_k(2,\R)$ (see Lemma \ref{convergence_rotation}).\\
\indent It is still unclear to us whether $(h,k)$-convergence groups are conjugate to  subgroups of $\PSL_k(2,\R)$, and it would be interesting to know whether there is a dynamical characterisation of groups that are topologically conjugate to subgroups of $\PSL_k(2,\R)$. Even though we do not have such a theorem, the notion will be central in our study of Lorentzian isometry groups.\\
\indent Note that the convergence property can be defined in any dimension, but the fact that convergence groups are topologically Fuchsian is specific to dimension one: there are convergence groups on $\Ss^2$ that are not topologically conjugate to subgroups of $\mathrm{O}(1,3)$ (see \cite{martin_skora}). However, the convergence property can still be useful in higher dimensions, such as in the proof of Ferrand-Obata's Theorem\footnote{A Riemannian manifold on which the conformal group acts non properly is conformally diffeomorphic to the round sphere or to the Euclidian space} (\cite{ferrand}, \cite{frances_tarquini}, \cite{frances07}).  The Lorentzian equivalent is false (see \cite{frances05b}), and this is related to the fact that Lorentzian conformal groups are not always convergence groups.

\subsection{Comments}
\subsubsection{Regularity} So far, we have not been precise on the regularity. The proofs that we present use geodesics (continuity of the exponential map) and the continuity of the curvature, which means that the Lorentz metrics have to be $C^2$-differentiable. However, everything points to thinking that Theorem \ref{main_theorem} still holds for continuous metrics. A special case (which corresponds to the case where $k=1$) is covered in \cite{Monclairb}. In this case we only need to show that a particular representation of the  isometry group (which is semi conjugate to $\rho_1^M$ and $\rho_2^M$) has the convergence property, so the extension of Theorem \ref{main_theorem} to continuous metrics is somehow linked to the question of knowing whether $(h,k)$-convergence groups are conjugate to subgroups of $\PSL_k(2,\R)$.
\subsubsection{Discrete groups versus connected groups} The  results classifying non compact isometry groups of compact Lorentz manifolds and actions of simple Lie groups mostly concern connected  groups (although some discrete cases are treated in \cite{Z99b} and \cite{PZ}), whereas our use of the convergence property allows us to deal simultaneously with connected groups and discrete groups. As we will see in  \ref{subsec:examples}, there is a much larger variety of discrete groups than connected groups (i.e. the existence of isometries does not imply the existence of a Killing field, not even of a local Killing field). 

\subsubsection{Rigid geometric structures}  We see a Lorentz metric on a surface as the data of two objects: a pair of transversal foliations (the  foliations by lightlike geodesics, they define the conformal structure), and a volume form (it fixes the metric in a conformal class). Foliations and volume forms are not rigid objects, in the sense that they can be preserved by an infinite dimensional group, but preserving both can generate some rigidity (the group preserving a Lorentz metric is always finite dimensional). We will find some results that concern the conformal group and some that concern only the isometry group. It can be translated in terms of the group that preserves one of the two objects (the foliations), and the group that preserves both.

\subsubsection{Regularity of the conjugacy} Even though Theorem \ref{main_theorem} and \ref{main_theorem_acausal} give a classification of the possible isometry groups for spatially compact Lorentz surfaces, they do not give any classification of those surfaces up to isometry. Such a classification would require a characterisation of $\rho_1^M,\rho_2^M$ up to conjugacy in $\Diff(\Ss^1)$. This problem is addressed \cite{Monclaira} in the case where $(M,g)$ is conformal to the De Sitter space $\dS_2$.

\subsubsection{Non spatially compact surfaces} The two representations $\rho_1^M,\rho_2^M$  are only  defined for spatially compact surfaces. If we consider globally hyperbolic surfaces that are not spatially compact, then we find representations in $\Diff(\R)$. For the universal covers of spatially compact surfaces, we find groups that are topologically conjugate to subgroups of the universal cover $\tilde \PSL(2,\R)$. The transverse geometry of Anosov flows points out that larger groups can act on globally hyperbolic surfaces. Given an Anosov flow $\p^t$ on a compact three-manifold $M$ that preserves a volume form $\omega$, the space $Q^{\p}$ of orbits of the lift of $\p^t$ to the universal cover $\tilde M$ is endowed with a Lorentz metric that is preserved by the action of the fundamental group $\pi_1(M)$.  Such flows exist on many hyperbolic manifolds  (see \cite{FH}). These examples belong to a special subcategory of Anosov flows: they are  $\R$-covered (\cite{B01}), which means that the Lorentz metric on $Q^\p$ is globally hyperbolic.\\
\indent Fundamental groups of hyperbolic three-manifolds are not subgroups of $\tilde \PSL(2,\R)$, so Theorem \ref{algebraic} cannot be simply extended to globally hyperbolic surfaces just by replacing $\PSL(2,\R)$ with $\tilde \PSL(2,\R)$.\\
\indent However, note that these examples all have low regularity: the differentiability of the Lorentz metric on $Q^\p$ is that of the stable and unstable foliations, hence only $C^{1+\e}$ (for all $\e <1$).

\subsubsection{Without global hyperbolicity} Finally, another possible generalisation would be to all Lorentz surfaces. In this case, the space of lightlike geodesics is a one dimensional non Hausdorff manifold (we can still define charts locally, but there is no condition on the global topology). This would lead to  group actions on one dimensional non Hausdorff manifolds, which has been studied (see \cite{B98}).

\subsection{Overview}  After giving the necessary Lorentzian background, we will start by defining two conformal models for spatially compact surfaces. The first is an immersion in the flat torus that has the advantage of being directly linked to the representations $\rho_1^M$ and $\rho_2^M$.  The second is an embedding in the flat cylinder, that will allow us to deal with manifolds that do not embed in the flat torus.\\
\indent In the third part, we describe the main examples of spatially compact surfaces with a non proper action of the isometry group. Next, we introduce the necessary tools of circle dynamics, the main one being convergence groups.\\
\indent In section \ref{sec:reducing_problem}, we show how it is sufficient to treat the case where $(M,g)$ embeds conformally in the flat torus. This will allow us to prove  Theorem \ref{main_theorem} in section \ref{sec:proof_main}. Finally, sections \ref{sec:elementary_acausal} and \ref{sec:non_elementary_acausal}  are dedicated to the proof of Theorem \ref{main_theorem_acausal}. Their length is due to the difficulty of manipulating the $(h,k)$-convergence property.


\section{Two conformal models}
\subsection{Lorentzian background} We will only use some basic notions of Lo\-rent\-zian geometry (for more details, see \cite{BEE}, \cite{hawking}, \cite{O'N}, or the introduction chapter of  \cite{o'neill_trous_noirs}). A Lorentz manifold is a manifold $M$ equipped with a symmetric $(2,0)$-tensor $g$ of signature $(-,+,\dots,+)$, called a \textbf{Lorentz metric}. If the regularity is not explicitly stated, we will assume Lorentz metrics to be $C^2$. A  tangent vector $u\in T_xM$ is said to be \textbf{timelike} (resp. \textbf{spacelike}, \textbf{lightlike}, \textbf{causal}) if $g_x(u,u)<0$ (resp. $g_x(u,u)>0$, $g_x(u,u)=0$, $g_x(u,u)\le 0$). A curve or a submanifold of $M$ is said to be timelike, spacelike, lightlike or causal if all its tangent vectors have the corresponding type.\\
\indent A \textbf{time orientation} of $(M,g)$ is  a timelike vector field. We will call a time oriented Lorentz manifold a \textbf{spacetime}. It allows us to define \textbf{future} (resp. \textbf{past}) \textbf{directed} causal vectors as vectors in the same connected component of the cone $g< 0$ (resp. in the other connected component). A future (resp. past) directed curve is a curve whose tangent vectors are future (resp. past) directed. Given a point $x\in M$, we define its future $J^+(x)$ (resp. its past $J^-(x)$) to be the set of endpoints of future (resp. past) curves starting at $x$. \\
\indent A conformal transformation is said to be \textbf{time orientation preserving} if it sends future directed vectors on future directed vectors. We will denote by $\Conf(M,g)$ (resp. $\Isom(M,g)$) the set of orientation and time orientation preserving conformal diffeomorphisms (resp. isometries). The isometry group is always a finite dimensional Lie Group (\cite{Adams}).\\
\indent A spacetime $(M,g)$ is said to be \textbf{globally hyperbolic}  if there is a topological  hypersurface $S\subset M$ (called a Cauchy hypersurface) such that every inextensible causal curve intersects $S$ exactly once. Smooth Cauchy hypersurfaces always exist (\cite{BS_surface_lisse}), moreover they are diffeomorphic to each other and $M$ is diffeomorphic to $\R\times S$. We say that it is \textbf{spatially compact} if it is globally hyperbolic and it has a compact Cauchy hypersurface. If $(M,g)$ is globally hyperbolic and $x,y\in M$, then $J^+(x)\cap J^-(y)$ is compact.\\
\indent Just as in the Riemannian case, a Lorentz metric defines a Levi-Civita connection, and we can define geodesics. They can be timelike, spacelike or lightlike. Unparametrised lightlike geodesics are preserved by conformal transformations. There is also a notion of curvature (and constant sectional curvature implies local isometry with a model space). \\
\indent Because of the connection, isometries are defined by their $1$-jet: if an isometry $f$ has a fixed point $x$ such that $Df_x$ is the identity map, then $f$ is the identity. This property is important as it is the reason for which a Lorentz metric can be considered a rigid geometric structure.\\
\indent In this paper, we will only deal with dimension two. Many questions about Lorentz surfaces have been studied, both in the compact case (e.g. completeness \cite{Carriere_Rozoy}, conjugate points \cite{bavard_mounoud}, closed geodesics \cite{suhr}) and in the non compact case (see \cite{Weinstein}).

\subsection{The lightlike foliations and actions on the circle}

\indent If $(M,g)$ is a time oriented Lorentz surface, then we can define two one dimensional foliations  $\F_1$ and $\F_2$ of $M$ given by the lightlike geodesics
(an orientation of $M\approx \R\times \Ss^1$ differentiates the two foliations).
Let $F_i=M/\F_i$ be the quotient and $p_i:M\to F_i$ the associated projection.  The restriction of  $p_i$ to a Cauchy hypersurface is a diffeomorphism.  Therefore, if $(M,g)$ is spatially compact, then
$F_i$ is diffeomorphic to $\Ss^1$.\\
\indent A conformal diffeomorphism sends (unparametrised) lightlike geodesics to lightlike geodesics, i.e. it preserves $\F_1\cup \F_2$. If it preserves the orientation, then it preserves $\F_1$ and $\F_2$, and it acts on the quotients, which gives two representations $\rho_i^M: \Conf(M,g)\to \Diff(F_i)\approx \Diff(\Ss^1)$, $i=1,2$.\\
\indent Note that $\rho_1^M,\rho_2^M$ are well defined up to conjugacy in $\Diff(\Ss^1)$.

\subsection{Immersion in the flat torus}
The map $x\mapsto (p_1(x),p_2(x))$ sends $M$ into $F_1\times F_2$. By choosing identifications between $F_1,F_2$ and $\Ss^1$, it gives us a map from $M$ to $\T^2$. We will denote by $p$ the map obtained after reversing the orientation on $F_1$ (i.e. $p(x)=(-p_1(x),p_2(x))$).

\begin{prop} \label{immersion_torus} The map $p:(M,[g]) \to (\T^2,[dxdy])$ is a conformal immersion.  If it is injective, then it is a conformal diffeomorphism onto its image.  \end{prop}

\begin{proof} The kernel of  $d_xp_i$ is the tangent space to $\F_i$
at $x$, therefore $\ker d_xp$ is the intersection of two transversal lines and is equal to $\{0\}$, and  $p$ is an immersion.  The image of isotropic vectors in $(M,g)$ being isotropic vectors for $(\T^2,dxdy)$,  we see that the metric $g$ is sent to a (not necessarily positive) multiple of $dxdy$.  Since we changed the orientation of $F_1$ in the definition of $p$,  the future is given by moving negatively along the $x$-axis and positively along the $y$-axis, i.e. the metric is conformal to $+dxdy$. \\ \indent By equality of dimensions, the immersion $p$ is an open map, and it only needs to be injective in order to be a diffeomorphism onto its image.
\end{proof}

We call $(\T^2,dxdy)$ the \textbf{flat Lorentzian torus}.\\
\indent The injectivity of $p$ is equivalent to the following property: two lightlike geodesics have at most one  intersection point (i.e. there are no null conjugate points). \\  \indent Since all globally hyperbolic   open sets of $\T^2$ satisfy this property, we see that $p$ is injective if and only if $(M,g)$ embeds conformally in $(\T^2,dxdy)$. This makes $p$ canonical in some sense: given a spatially compact surface that embeds conformally in $\T^2$, there are many conformal embeddings, but we can choose one that satisfies certain properties.

\subsection{Embedding in the flat cylinder} \label{subsec:flat_cylinder}
\subsubsection{Definition} The problem that we encounter with the previous conformal model is that it is not always an embedding, but only an immersion. However, we will now see that there is another conformal model that always give an embedding.\\
\indent The \textbf{flat Lorentzian cylinder} is the quotient of the Minkowski plane by a spacelike translation (note that these quotients are not all isometric to each other, but they are conformally equivalent).

\begin{theo} \label{flat_cylinder} All spatially compact surfaces embed conformally in the flat Lorentzian cylinder. \end{theo}
\begin{proof} Let $(M,g)$ be a spatially compact surface, and consider its universal cover $(\widetilde M,\widetilde g)$. It  also has two foliations by lightlike geodesics $\tilde \F_1,\tilde \F_2$, and the quotients $\tilde F_1,\tilde F_2$ are diffeomorphic to the real line $\R$. This gives us a conformal immersion $\tilde p:( \widetilde M, [\widetilde g])\to (\R^2,[dxdy])$. This time, however, it is always an embedding: two distinct lightlike geodesics on a simply connected Lorentz surface intersect  at most once (see p.51 in \cite{Weinstein}). \\
\indent Let $F\in \Isom(\widetilde M,\widetilde g)$ be a generator of $\pi_1(M)$. It is a conformal diffeomorphism for $[dxdy]$, therefore it can be written $F(x,y)=(f_1(x),f_2(y))$  for some $f_1,f_2\in \Diff(\R)$. Since the quotient of $\widetilde M$ by $F$ is causal, we see that $(x,y)\in p(\widetilde M)$ and $F(x,y)$ are not causally related (i.e. they are not in the future or in the past of each other). This shows that  either $f_1(x)>x$ and $f_2(y)>y$, either $f_1(x)<x$ and $f_2(y)<y$ (for one $(x,y)$, hence for all $(x,y)$ because $\widetilde M$ is connected). So up to replacing $F$ by $F^{-1}$, we can assume that $f_1(x)>x$ and $f_2(y)>y$ for all $x,y\in \R$. This implies that they are both differentially conjugate to the translation $x\mapsto x+1$, and we can assume that $F(x,y)=(x+1,y+1)$. This shows that $(M,g)$ embeds in the quotient of $\R^2$ by the map $F$, which is the flat Lorentzian cylinder. \end{proof}
Even though this conformal model gives an embedding for $M$, we will mostly use the map $\tilde p$ defined on the universal cover $\tilde M$.

\subsubsection{Conformal classification} We have a simple characterisation of the image in $\R^2$. We call an open set $U\subset \R^2$ \textbf{canonically embedded} if there is a spatially compact  surface $( M, g)$ such that $U=\tilde p(\widetilde M)$.
 \begin{prop} Let $U\subset \R^2$ be a  canonically embedded open set.  There are non decreasing maps $\tilde h_\gets, \tilde h_{\downarrow} :\R\to\R\cup\{-\infty\}$ and $\tilde h_\to, \tilde h_{\uparrow}:\R\to\R\cup\{+\infty\}$    that commute with the translation $x\mapsto x+1$ such that :
 \begin{eqnarray*} U &= & \{(x,y)\in \R^2 \vert \tilde h_{\downarrow}(x)<y<\tilde h_{\uparrow}(x)\}\\ & =&  \{(x,y)\in \R^2 \vert \tilde h_{\gets}(y)<x<\tilde h_{\to}(y)\}
 \end{eqnarray*} \end{prop}
\begin{proof} 
\indent Every $x\in \R$ defines a vertical line in $U$: there are $\tilde h_{\downarrow}(x)\in \intfo{-\infty}{+\infty}$ and $\tilde h_{\uparrow}(x)\in \intof{-\infty}{+\infty}$ such that: $$\{x\}\times \R\cap U= \{x\}\times \intoo{\tilde h_{\downarrow}(x)}{\tilde h_{\uparrow}(x)}$$ \indent  Since $U$ is invariant under the map $(x,y)\mapsto (x+1,y+1)$, we see that $\tilde h_{\uparrow}(x+1)=\tilde h_{\uparrow}(x)+1$ and $\tilde h_\downarrow (x+1) =\tilde h_\downarrow(x) +1$ for all $x\in \R$. \\
\indent Assume that there is $x_0\in \R$ such that $\tilde h_{\downarrow}(x_0)\ne -\infty$. Let $x>x_0$, and assume by contradiction that $(x,\tilde h_{\downarrow}(x_0))\in U$. If $y>\tilde h_{\downarrow}(x_0)$,  the diamond $J^+(x,\tilde h_{\downarrow}(x_0))\cap J^-(x_0,y)$ contains $(x_0,\tilde h_{\downarrow}(x_0))$ and is not included in $U$ (see Figure \ref{fig:proof_cylinder_model}), which implies that $U$ is not globally hyperbolic and is absurd. Hence $(x,\tilde h_{\downarrow}(x_0))\notin U$ and $\tilde h_{\downarrow}(x)\ge \tilde h_{\downarrow}(x_0)>-\infty$. This shows that $\tilde h_{\downarrow}$ is non decreasing. Reversing the time orientation shows that $\tilde h_{\uparrow}$ is also non decreasing.\\
\indent By exchanging the roles of $x$ and $y$, we define $\tilde h_\gets$ and $\tilde h_\to$ in the same way.
\begin{figure}[h]
\definecolor{uuuuuu}{rgb}{0.27,0.27,0.27}
\begin{tikzpicture}[line cap=round,line join=round,>=triangle 45,x=1.0cm,y=1.0cm]
\clip(-6.3,-1) rectangle (7.3,5.3);
\draw[smooth,samples=100,domain=-2.8:-1] plot(\x,{-1.5-(\x)-0.2*sin((5*(\x))*180/pi)});
\draw [->] (-3,0) -- (-3,5);
\draw [->] (-3,0) -- (3,0);
\draw [dash pattern=on 3pt off 3pt] (-2.18,4)-- (2,4);
\draw [dash pattern=on 3pt off 3pt] (-2.18,0.48)-- (-2.18,4);
\draw [dash pattern=on 3pt off 3pt] (-2.18,0.48)-- (2,0.48);
\draw [dash pattern=on 3pt off 3pt] (2,4)-- (2,0.48);
\draw (-2.4,0) node[anchor=north west] {$x_0$};
\draw (1.8,0) node[anchor=north west] {$x$};
\draw (-1.1,-0.2) node[anchor=north west] {$Gr(\tilde h_{\downarrow})$};
\draw (-3,0.75) node[anchor=north east] {$\tilde h_{\downarrow}(x_0)$};
\draw (-3,4.25) node[anchor=north east] {$y$};
\begin{scriptsize}
\draw [color=uuuuuu] (-3,4)-- ++(-1.0pt,0 pt) -- ++(2.0pt,0 pt) ++(-1.0pt,-1.0pt) -- ++(0 pt,2.0pt);
\draw [color=uuuuuu] (-3,0.48)-- ++(-1.0pt,0 pt) -- ++(2.0pt,0 pt) ++(-1.0pt,-1.0pt) -- ++(0 pt,2.0pt);
\draw [color=uuuuuu] (-2.18,0)-- ++(-1.0pt,0 pt) -- ++(2.0pt,0 pt) ++(-1.0pt,-1.0pt) -- ++(0 pt,2.0pt);
\draw [color=uuuuuu] (2,0)-- ++(-1.0pt,0 pt) -- ++(2.0pt,0 pt) ++(-1.0pt,-1.0pt) -- ++(0 pt,2.0pt);
\end{scriptsize}
\end{tikzpicture}
\caption{$\tilde h_{\downarrow}$ is non decreasing} \label{fig:proof_cylinder_model}
\end{figure}
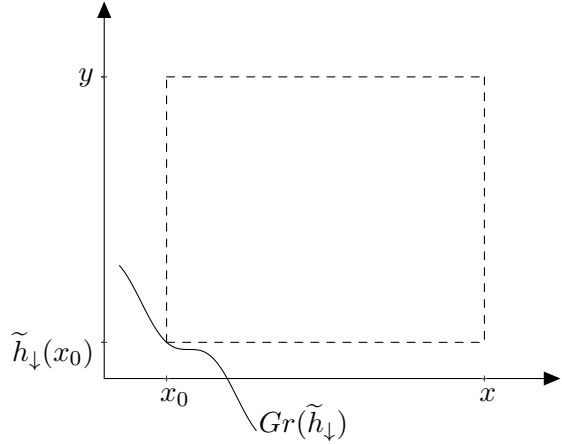
\end{proof}
\indent This implies that a spatially compact surface is conformally diffeomorphic to an open set of the flat cylinder delimited by non timelike curves. Note that if there is $x_0\in \R$ such that $\tilde h_{\downarrow}(x_0)=-\infty$ (resp. $\tilde h_{\uparrow}(x_0)=+\infty$), then $\tilde h_{\downarrow}(x)=-\infty$ (resp. $\tilde h_{\uparrow}(x)=+\infty$) for all $x\in \R$.\\
\indent The map $\tilde h_\uparrow$ (resp. $\tilde h_\downarrow$) is always left continuous (resp. right continuous). All such maps can be obtained: given $\tilde h_-,\tilde h_+ : \R \to \R$ non decreasing  that commute with $x\mapsto x+1$, such that $\tilde h_- < \tilde h_+$ and $\tilde h_+$ (resp. $\tilde h_-$) is left continuous (resp. right continuous), we obtain a spatially compact surface $M$ whose universal cover is the set of points  $(x,y)\in \R^2$ such that $\tilde h_-(x)<y<\tilde h_+(x)$. It is unique up to a conformal diffeomorphism. \\
\indent The boundary $\partial U\subset \R^2$ is the union of the graphs of $\tilde h_{\downarrow}$ and $\tilde h_{\uparrow}$ and of the vertical lines joining discontinuities (we set the graphs of the constants $\pm \infty$ to be the empty set). We can define the conformal boundary $\partial M$ of a spatially compact surface $(M,g)$ to be the projection of the boundary of $U$ in the flat cylinder. It is an achronal set (two points cannot be joined by a timelike curve). However, if $\tilde h_{\downarrow}$ or $\tilde h_{\uparrow}$ is constant on an interval, then the boundary can contain causal curves. Such causal curves in the boundary can only arise when $\tilde h_{\downarrow}$ and $\tilde h_{\uparrow}$ fail to be homeomorphisms.

\begin{defi} Let $(M,g)$ be a spatially compact surface. We say that the conformal boundary is acausal if the maps $\tilde h_{\downarrow}, \tilde h_\uparrow$ are either homeomorphisms of $\R$ either infinite. \end{defi}

Note that this boundary falls in the general concept of conformal boundary for spacetimes (see \cite{FHS}). It is a general fact that the conformal boundary of a globally hyperbolic spacetime is achronal, but not necessarily acausal.

\subsubsection{From the cylinder to the torus}  Given a spatially compact surface $(M,g)$, we have defined a conformal embedding $\tilde p : \tilde M\to \R^2$ and a conformal immersion $p : M\to \T^2$. If we denote by $\pi : \tilde M\to M$ the (Lorentzian) universal cover and by $\pi_0 : \R^2\to \T^2=\R^2/\Z^2$ the universal cover of the flat torus, then it follows from the definitions that $\pi_0\circ \tilde p = p\circ \pi$.\\
\indent The projections of  $(x,y)$ and $(x,y+1)$ in the flat cylinder are different points on the same lightlike geodesics.  This shows that $M$ embeds conformally in the torus if and only if $\tilde h_{\uparrow}(x)-\tilde h_{\downarrow}(x)\le 1$ for all $x\in \R$. In this case, if we denote by $h_{\uparrow}, h_\downarrow$ the associated maps on the circle, then we find that $p(M)=\{ (x,y)\in \Ss^1\times \Ss^1 \vert h_{\downarrow}(x)<y<h_{\uparrow}(x)\le h_{\downarrow}(x)\}$. We also see that if the conformal boundary of $(M,g)$ is acausal, then $h_{\uparrow}$ and $ h_\downarrow$ are circle homeomorphisms.
\begin{prop} Let $(M,g)$ be a spatially compact surface with an acausal conformal boundary. If $(M,g)$ embeds conformally in the torus, then the representations $\rho_1^M$ and $\rho_2^M$ are topologically conjugate.
\end{prop}
\begin{proof} Given $\p\in \Conf(M,g)$ and $(x,y)\in p(M)$, we can write: $$p\circ \p \circ p^{-1}(x,y)=(\rho_1^M(\p)(x),\rho_2^M(\p)(y))$$ Since for all $x\in \Ss^1$ there is $y\in \Ss^1$ such that $(x,y)\in p(M)$, the fact that $p\circ \p\circ p^{-1}$ preserves $\partial p(M)$ implies that $\rho_2^M\circ h_{\downarrow} = h_{\downarrow} \circ \rho_1^M$ and $\rho_2^M\circ h_{\uparrow} = h_{\uparrow} \circ \rho_1^M $.
\end{proof}

Since the maps $\tilde h_\uparrow, \tilde h_\downarrow, \tilde h_\to, \tilde h_\gets$ commute with $x\mapsto x+1$, they define maps on $\Ss^1$ as soon as they are finite, which is the case when $(M,g)$ embeds in the torus. Let $h_\uparrow, h_\downarrow, h_\to, h_\gets$ be the induced maps on the circle. Using the invariance of $\tilde p(\tilde M)\subset \R^2$ by lifts of conformal diffeomorphisms, we obtain the following relations:

\begin{prop} Let $(M,g)$ be a spatially compact surface that embeds conformally in $\T^2$. If $\p\in \Conf(M,g)$, then the maps $h_\uparrow, h_\downarrow, h_\to, h_\gets$ satisfy:
\begin{eqnarray*} \rho_2^M(\p) \circ h_\downarrow &=& h_\downarrow \circ \rho_1^M(\p) \\
\rho_2^M(\p) \circ h_\uparrow &=& h_\uparrow \circ \rho_1^M(\p) \\
\rho_1^M(\p) \circ h_\gets &=& h_\gets \circ \rho_2^M(\p) \\
\rho_1^M(\p) \circ h_\to &=& h_\to \circ \rho_2^M(\p) 
\end{eqnarray*}
\end{prop}

Note that when the conformal boundary is acausal, we automatically have $h_\to = h_\downarrow^{-1}$ and $h_\gets=h_\uparrow^{-1}$.

\subsubsection{Link with $(G,X)$-structures} What is at the heart of the two conformal models that we have defined is the fact that surfaces are conformally flat. It is a simple observation for Lorentz surfaces, and a Theorem of Gauss for Riemannian surfaces. In general, it is only a local condition, i.e. we have local conformal diffeomorphisms with the Minkowski space, but we have shown that we can find a global embedding for globally hyperbolic surfaces. This translates in terms of completeness of a $(G,X)$-structure.\\
\indent If $G$ is a group acting by diffeomorphisms on a simply connected manifold $X$, then a $(G,X)$-structure on a manifold $M$ is an atlas of local diffeomorphisms with $X$, the transitions maps being elements of $G$. If the action is analytic (i.e. elements of $G$ are uniquely defined by their action on a small open set), then we can define a holonomy morphism $h:\pi_1(M)\to G$ and an equivariant map $D:\tilde M \to X$ called the developing map. A $(G,X)$-structure is said to be complete if $D$ is injective (which means that $M$ is a quotient of an open set of $X$).\\
\indent What is interesting here is that we have shown completeness for a structure where even the existence of the developing map is not given by the general theory (the action of the conformal group is not analytic).

\section{Examples}
\subsection{The constant curvature model spaces}
A particularity of two dimensional Lorentzian geometry is that, from the isometry group point of view, there are only two constant curvature geometries. Indeed, if $(M,g)$ is a Lorentz surface, then $-g$ is also a Lorentz metric. If $K_g$ is the curvature of $g$, then the curvature of $-g$ is $K_{-g}=-K_g$. This shows that the positive and negative constant curvature spaces have the same isometry group.\\
\indent The flat simply connected model is the Minkowski space $\R^{1,1}=(\R^2,dxdy)$. The spatially compact model is the flat cylinder: it is the quotient of $\R^{1,1}$ by a spacelike translation. The second conformal model described above shows that any spatially compact surface is conformal to an open set of the flat cylinder bounded by two spacelike curves.\\
\indent As we mentioned earlier, the positive and negative constant curvature spaces share the same isometry group. However,  the causality of a metric $g$ does not imply the causality of $-g$. In the classic model spaces, the positive model is the one sheeted hyperboloïd in $\R^{1,2}$. It is spatially compact. But the negative curvature model is not causal. For this reason, we will only consider the positive curvature model.\\
\indent We consider $\R\proj^1=\R\cup \{\infty\}$ and we endow $\rm{dS}_2 = \R\proj^1\times \R\proj^1\setminus \Delta$ with the Lorentz metric $g_{\rm{dS}_2}=\frac{4dxdy}{(x-y)^2}$. It is preserved by the diagonal action of $\PSL(2,\R)$ acting projectively on $\R\proj^1$ (i.e. the matrix $\begin{pmatrix}  a& b   \\  c& d \end{pmatrix}$ acts by the homography $x\mapsto \frac{ax+b}{cx+d}$), and it has constant curvature $1$. It is isometric to the one sheeted hyperboloïd, and the isometry gives the classic isomorphism $\PSL(2,\R)\approx \mathrm{SO}^\circ(1,2)$ (see \cite{DG}).
\subsection{Open sets of $\rm{dS}_2$ and subgroups of $\PSL(2,\R)$} \label{subsec:examples}
We wish to understand which spatially compact surfaces $(M,g)$ can satisfy $\rho_1^M=\rho_2^M$ and $\rho_1^M(\Isom(M,g))\subset \PSL(2,\R)$. We will first focus on open sets of $\rm{dS}_2$ and conformal changes of the metric on these open sets.\\
\indent Given  $h\in \Homeo(\Ss^1)$,  the connected components of  $\Ss^1\times \Ss^1\setminus (\Delta \cup Gr(h))$ are globally hyperbolic (because $h$ is orientation preserving), so they have two possible topologies: either a plane or a cylinder. If $h$ has no fixed point, then $\Ss^1\times \Ss^1\setminus (\Delta \cup Gr(h))$ has two connected components, both spatially compact. If $h$ has at least one fixed point, then there is one  spatially compact component, which we denote it by $\Omega_h$ (see Figure \ref{fig:omega_h}). The action of  a group $\Gamma\subset \PSL(2,\R)$ preserves $\Omega_h$ if and only $h$ commutes with every element of $\Gamma$. This shows that looking for open sets preserved by a subgroup of $\PSL(2,\R)$ can be achieved by constructing homeomorphisms that commute with every element of the group. Note that $\Gamma$ acts non properly on $\Omega_h$ if and only if it acts non properly on $\rm{dS}_2$, which is equivalent to $\Gamma$ not being relatively compact (this will be shown in Proposition \ref{proper_compact_group}).\\
\begin{figure}[h]
\definecolor{xdxdff}{rgb}{0.49,0.49,1}
\definecolor{qqqqff}{rgb}{0,0,1}
\definecolor{cqcqcq}{rgb}{0.75,0.75,0.75}
\begin{tikzpicture}[line cap=round,line join=round,>=triangle 45,x=1.0cm,y=1.0cm,scale=0.5]
\clip(-9.3,-4) rectangle (16.9,6.3);
\fill[fill=black,pattern=vertical lines] (-2,-4) -- (-1,0) -- (1,2) -- (5,3) -- cycle;
\fill[fill=black,pattern=vertical lines] (5,3) -- (7,4) -- (8,6) -- cycle;
\draw (-2,6)-- (8,6);
\draw (8,6)-- (8,-4);
\draw (8,-4)-- (-2,-4);
\draw (-2,-4)-- (-2,6);
\draw (-2,-4)-- (8,6);
\draw (5.14,3.48) node[anchor=north west] {$Gr(h)$};
\draw (0.76,4.74) node[anchor=north west] {$\Omega_h$};
\draw (-2,-4)-- (-1,0);
\draw (-1,0)-- (1,2);
\draw (1,2)-- (5,3);
\draw (5,3)-- (-2,-4);
\draw (5,3)-- (7,4);
\draw (7,4)-- (8,6);
\draw (8,6)-- (5,3);
\end{tikzpicture}
\caption{$\Omega_h$} \label{fig:omega_h}
\end{figure}
\indent If $g_{\sigma}=e^{\sigma} g_{\rm{dS}_2}$, then the action of $\Gamma\subset \PSL(2,\R)$ is isometric for $g$ if and only if $\sigma$ is invariant under $\Gamma$.\\
\indent Note that all subgroups cannot have a non trivial commuting homeomorphism or invariant function. If $\Gamma=\PSL(2,\Z)$ (or any lattice), then it has dense orbits on $\rm{dS}_2$, which shows that any invariant function is constant, and the only commuting homeomorphism is the identity. The important notion is the \textbf{limit set}. If $\Gamma \subset \PSL(2,\R)$ is a subgroup, then let $L_{\Gamma}\subset \Ss^1$ be a  minimal closed $\Gamma$-invariant subset of $\Ss^1$. There are three possibilities for $L_{\Gamma}$: it can be a finite set, a Cantor set, or $\Ss^1$. We say that $\Gamma$ is \textbf{non elementary} if $L_{\Gamma}$ is infinite. In this case, it is uniquely defined, and it is equal to the intersection of $\Ss^1=\partial_{\infty}\h^2$ with $\overline{\Gamma.x}$ for all $x\in \h^2$.

\begin{prop} \label{homeo_commutes} Let $\Gamma \subset \PSL(2,\R)$ be a subgroup such that $L_{\Gamma}\ne \Ss^1$. Then there is $h\in \Homeo(\Ss^1)\setminus \{Id\}$ that commutes with every element of $\Gamma$. \end{prop}

\begin{proof} We write $\Ss^1\setminus L_{\Gamma}=\bigcup_{i\in \N} I_i$ as the union of its connected components. This induces an action of $\Gamma$ on $\N$, and let $R\subset \N$ be a fundamental domain. For $i\in R$, we set $h_{/I_i}$ to be a homeomorphisms fixing the endpoints of $I_i$ that commutes with the elements of $\Gamma$ stabilizing $I_i$ and that is not the identity. For $\gamma \in \Gamma$, we set $h_{/\gamma (I_i)}= \gamma \circ h_{/I_i}\circ \gamma^{-1}$. It is well defined because if $\gamma_1( I_i)\cap \gamma_2 (I_i)\ne \emptyset$, then $\gamma_2^{-1}\gamma_1$ stabilizes $I_i$, hence commutes with $h_{/I_i}$, and $\gamma_1 \circ h_{/I_i}\circ \gamma_1^{-1}=\gamma_2 \circ h_{/I_i}\circ \gamma_2^{-1}$.\\
\indent It is continuous because it is the identity on $L_{\Gamma}$, and it commutes with all elements of $\Gamma$.
\end{proof}

If $h\ne Id$, then $\Omega_h$ is not conformal to $\dS_2$, so we cannot expect all the Lorentz surfaces under study to be conformal to $\dS_2$, even if the group is non elementary.\\
\indent It is easy to see that the conformal boundary is not necessarily acausal: if $\Gamma\subset \PSL(2,\R)$ is non elementary and $L_\Gamma \ne \Ss^1$, then write $\Ss^1\setminus L_\Gamma =\bigcup\suite \intoo{a_n}{b_n}$, and let $h$ be the identity on $L_\Gamma$ and $h(x)=b_n$ for $x\in \intoo{a_n}{b_n}$. It is a non decreasing map of degree one that commutes with $\Gamma$, and it bounds an open set of $\dS_2$ invariant by $\Gamma$.

\begin{prop} \label{invariant_function} Let $\Gamma\subset \PSL(2,\R)$ be a subgroup such that $L_{\Gamma}\ne \Ss^1$. Then there is a non constant smooth $\Gamma$-invariant function $\sigma: \rm{dS}_2\to \R$. \end{prop}

\begin{proof}  Start by writing $\Ss^1\setminus L_{\Gamma}= \bigcup_{i\in \N}I_i$ as the union of its connected components. We start by setting $\sigma=0$ on $(L_{\rho(\Gamma)}\times \Ss^1\cup \Ss^1\times L_{\rho(\Gamma)})\setminus \Delta$ and on $I_i\times I_i\setminus \Delta$ for $i\in \N$. For $x\in I_i\times I_j$ with $i\ne j$, consider $R_1,R_2,R_3,R_4$ the four rectangles that have $x$ as one corner and a corner of $I_i\times I_j$ as the opposite corner (see Figure \ref{fig:invariant_functions}). Let $\sigma(x)=\omega(R_1)\omega(R_2)\omega(R_3)\omega(R_4)$ where $\omega$ is the volume form associated to the de Sitter metric. By using the explicit formula $\omega(\intff{a}{b}\times\intff{c}{d})=4\Log([a,b,c,d])$ where $[a,b,c,d]=\frac{a-c}{a-d}\frac{b-d}{b-c}$ is the cross-ratio, we see that $\sigma$ is continuous. The function $\sigma$ is smooth in the interior of rectangles $I_i\times I_j$, i.e. where it is non zero. If $F:\R \to \R$ is smooth and constant on a neighbourhood of $0$ sufficiently small so that $F\circ \sigma$ is not constant, then $F\circ \sigma$ is  $\Gamma$-invariant, non constant and smooth.\\
\indent There are many other ways of constructing invariant functions. We could set $\sigma(x)$ on $I_i\times I_i$ to be $F(\omega(R ))$ where $R$ is the rectangle amongst $R_1,R_2,R_3$ and $R_4$ defined above that is included in $\Ss^1\times \Ss^1\setminus \Delta$ (see Figure \ref{fig:invariant_functions}).\\
\indent Finally, we could also choose $\sigma$ arbitrarily on the squares $I_i\times I_i$ where $i$ is in a fundamental domain for the action of $\Gamma$ on the connected components of $\Ss^1\setminus L_\Gamma$, and let $\sigma$ be constant on rectangles $I_i\times I_j$ with $i\ne j$.
\end{proof}

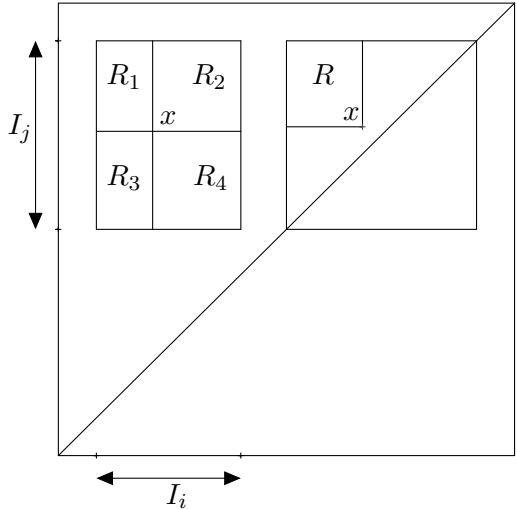
\begin{figure}[h]
\definecolor{uuuuuu}{rgb}{0.27,0.27,0.27}
\begin{tikzpicture}[line cap=round,line join=round,>=triangle 45,x=1.0cm,y=1.0cm]
\clip(-6.5,-1.9) rectangle (11.6,5.2);
\draw (-3,5)-- (3,5);
\draw (3,5)-- (3,-1);
\draw (3,-1)-- (-3,-1);
\draw (-3,-1)-- (-3,5);
\draw (-3,-1)-- (3,5);
\draw (-2.5,4.5)-- (-0.6,4.5);
\draw (-0.6,4.5)-- (-0.6,2);
\draw (-0.6,2)-- (-2.5,2);
\draw (-2.5,4.5)-- (-2.5,2);
\draw (-2.5,3.3)-- (-0.6,3.3);
\draw (-1.76,2)-- (-1.76,4.5);
\draw [<->] (-3.3,2) -- (-3.3,4.5);
\draw [<->] (-0.6,-1.3) -- (-2.5,-1.3);
\draw (-1.72,-1.24) node[anchor=north west] {$I_i$};
\draw (-3.8,3.66) node[anchor=north west] {$I_j$};
\draw (-2.5,4.32) node[anchor=north west] {$R_1$};
\draw (-1.38,4.32) node[anchor=north west] {$R_2$};
\draw (-2.5,2.98) node[anchor=north west] {$R_3$};
\draw (-1.36,2.98) node[anchor=north west] {$R_4$};
\draw (-1.8,3.7) node[anchor=north west] {$x$};
\draw (0,4.5)-- (0,2);
\draw (0,2)-- (2.5,2);
\draw (2.5,2)-- (2.5,4.5);
\draw (2.5,4.5)-- (0,4.5);
\draw (1,4.5)-- (1,3.36);
\draw (0,3.36)-- (1,3.36);
\draw (1.1,3.75) node[anchor=north east] {$x$};
\draw (0.2,4.32) node[anchor=north west] {$R$};
\begin{scriptsize}
\draw [color=black] (-2.5,-1)-- ++(-1.0pt,0 pt) -- ++(2.0pt,0 pt) ++(-1.0pt,-1.0pt) -- ++(0 pt,2.0pt);
\draw [color=black] (-0.6,-1)-- ++(-1.0pt,0 pt) -- ++(2.0pt,0 pt) ++(-1.0pt,-1.0pt) -- ++(0 pt,2.0pt);
\draw [color=black] (-3,2)-- ++(-1.0pt,0 pt) -- ++(2.0pt,0 pt) ++(-1.0pt,-1.0pt) -- ++(0 pt,2.0pt);
\draw [color=black] (-3,4.5)-- ++(-1.0pt,0 pt) -- ++(2.0pt,0 pt) ++(-1.0pt,-1.0pt) -- ++(0 pt,2.0pt);
\draw [color=uuuuuu] (1,3.36)-- ++(-1.0pt,0 pt) -- ++(2.0pt,0 pt) ++(-1.0pt,-1.0pt) -- ++(0 pt,2.0pt);
\end{scriptsize}
\end{tikzpicture}
\caption{Construction of invariant functions} \label{fig:invariant_functions}
\end{figure}

This result implies that the curvature of a Lorentz surface with a non proper action of the isometry group is not necessarily constant.\\
\indent  By choosing the function $\sigma$ arbitrarily on  the rectangles $I_i\times I_j$ where $(i,j)$ lies in a fundamental domain for the action of $\Gamma$ on $\N^2$  and extending by invariance (and always composing with a cut-off function in $\R$ to get smoothness), we could show that there are Lorentz metrics such that all the local isometries are restrictions of elements of $\Gamma$ (by showing that such a metric is generic in the set of $\Gamma$-invariant metrics, following the proof of Sunada \cite{sunada} for Riemannian manifolds).\\
\indent Note that the same method would work on $\Omega_h$ and could produce an invariant metric $g_{\sigma}$ that is not extendible to $\Ss^1\times \Ss^1\setminus \Delta$, so maximal spatially compact surfaces can have a non trivial conformal boundary, eventually  not acausal.
\subsection{Finite covers} If $(M,g)$ is a spatially compact surface, then so are its finite covers. By applying this to the De Sitter space, we find space-times for which the representations $\rho_1^M$, $\rho_2^M$  have values in  $\PSL_k(2,\R)$ (the order $k$ covering of $\PSL(2,\R)$, its elements are the lifts of elements of $\PSL(2,\R)$ to the order $k$ covering of $\Ss^1$, which is still a circle).\\
\indent By taking finite covers of the examples constructed above, we obtain groups that are covers of subgroups of $\PSL(2,\R)$. However, by using similar constructions starting with the $k$ covering of $\rm{dS}_2$, we obtain subgroups of $\PSL_k(2,\R)$ that are not covers of subgroups of $\PSL(2,\R)$.
\subsection{Extensions of $\rm{dS}_2$} A common property to the examples that we have defined so far is that they all embed conformally in the flat torus. However, it is not always the case.

\begin{prop} There are spatially compact surfaces $(M,g)$ with a non proper action of $\Isom(M,g)$ that do not embed conformally in $\T^2$. \end{prop}

\begin{proof} Start with $\gamma \in \PSL(2,\R)$ a parabolic element. Let $x_0\in \Ss^1$ be its fixed point. We consider the open sets $U=\{(x,y)\in \Ss^1\times \Ss^1 \vert x_0<y<x<x_0\}$ and $V=\{(x,y)\in \Ss^1\times \Ss^1\vert x_0<\gamma(x)<y<x_0 \}$. Up to replacing $\gamma$ by $\gamma^{-1}$, we can assume that $\Delta\setminus \{(x_0,x_0)\} \subset V$.  Let $g_U$ be the restriction to $U$ of the De Sitter metric. Let $g_V$ be a metric on $V$ preserved by $\gamma$ that is equal to the De Sitter metric in a neighbourhood of the axes of $\gamma$ (such a metric exists because $\gamma$ acts properly on the complement of the axes). Let $M$ be the manifold obtained by gluing $U$ and $V$ along $\{x_0\}\times \Ss^1\cup \Ss^1\times \{x_0\}$. Since the metrics $g_U$ and $g_V$ are equal on a neighbourhood of the glued parts, they endow $M$ with a Lorentz metric $g$  that is preserved by $\gamma$. The lightlike geodesics leaving from a point of $U\cap V$ meet again, and $\gamma$ acts non properly on $M$. The graph of any rotation $R_{\alpha}$ is a Cauchy hypersurface in $M$, therefore it is spatially compact. 
\end{proof}

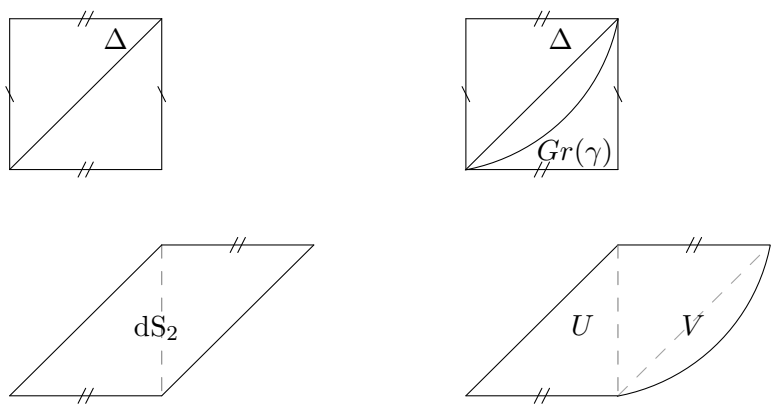
\begin{figure}[h]
\definecolor{aqaqaq}{rgb}{0.63,0.63,0.63}
\definecolor{yqyqyq}{rgb}{0.5,0.5,0.5}
\begin{tikzpicture}[line cap=round,line join=round,>=triangle 45,x=1.0cm,y=1.0cm]
\clip(-4.69,0.5) rectangle (6.57,6.2);
\draw [dash pattern=on 5pt off 5pt,color=yqyqyq] (-2,3)-- (-2,1);
\draw (-4,1)-- (-2,1);  \draw (-3.1,0.9)-- (-3,1.1); \draw (-3,0.9)-- (-2.9,1.1);
\draw (0,3)-- (-2,3);  \draw (-1.1,2.9)-- (-1,3.1); \draw (-1,2.9)-- (-0.9,3.1);
\draw (0,3)-- (-2,1);
\draw (-2,3)-- (-4,1);
\draw (2,1)-- (4,1); \draw (2.9,0.9)-- (3,1.1); \draw (3,0.9)-- (3.1,1.1);
\draw (4,3)-- (2,1);
\draw [dash pattern=on 5pt off 5pt,color=aqaqaq] (4,1)-- (6,3);
\draw (6,3)-- (4,3); \draw (4.9,2.9)-- (5,3.1); \draw (5,2.9)-- (5.1,3.1);
\draw [dash pattern=on 5pt off 5pt,color=yqyqyq] (4,3)-- (4,1);
\draw [shift={(3.52,3.48)}] plot[domain=4.9:6.09,variable=\t]({1*2.53*cos(\t r)+0*2.53*sin(\t r)},{0*2.53*cos(\t r)+1*2.53*sin(\t r)});
\draw (-2.5,2.2) node[anchor=north west] {$\rm{dS}_2$};
\draw (3.25,2.2) node[anchor=north west] {$U$};
\draw (4.7,2.2) node[anchor=north west] {$V$};
\draw (-4,6)-- (-2,6);  \draw (-3.1,5.9)-- (-3,6.1); \draw (-3,5.9)-- (-2.9,6.1);
\draw (-2,4)-- (-4,4);  \draw (-3.1,3.9)-- (-3,4.1); \draw (-3,3.9)-- (-2.9,4.1);
\draw (-4,6)-- (-4,4); \draw(-4.05,5.1)-- (-3.95,4.9);
\draw (-2,6)-- (-2,4); \draw(-2.05,5.1)-- (-1.95,4.9);
\draw (2,6)-- (2,4); \draw(1.95,5.1)-- (2.05,4.9);
\draw (4,4)-- (4,6); \draw(3.95,5.1)-- (4.05,4.9);
\draw (2,6)-- (4,6); \draw (2.9,5.9)-- (3,6.1); \draw (3,5.9)-- (3.1,6.1);
\draw (4,4)-- (2,4); \draw (2.9,3.9)-- (3,4.1); \draw (3,3.9)-- (3.1,4.1);
\draw [shift={(1.52,6.48)}] plot[domain=4.9:6.09,variable=\t]({1*2.53*cos(\t r)+0*2.53*sin(\t r)},{0*2.53*cos(\t r)+1*2.53*sin(\t r)});
\draw (4,6)-- (2,4);
\draw (-2,6)-- (-4,4);
\draw (-2.9,5.99) node[anchor=north west] {$\Delta$};
\draw (2.95,5.99) node[anchor=north west] {$\Delta$};
\draw (2.8,4.55) node[anchor=north west] {$Gr(\gamma)$};
\end{tikzpicture}
\caption{Extension of $\rm{dS}_2$} \label{fig:extension_de_sitter}
\end{figure}

We constructed this example with a parabolic element, but the same method would apply with a hyperbolic element, or more generally with any subgroup $\Gamma\subset \PSL(2,\R)$ such that $L_{\Gamma}\ne \Ss^1$.

\section{Actions on the circle}
\subsection{Minimal closed invariant subsets of $\Ss^1$}  An important object in the study of  groups of circle homeomorphisms is a minimal closed invariant set. It extends the notion of limit set for subgroups of $\PSL(2,\R)$. Given a group $G \subset \Homeo(\Ss^1)$,  exactly one of the following conditions is satisfied (see \cite{Gh01} for a proof and more detail):
\begin{enumerate} \item $G$ has a finite orbit \item All orbits of $G$ are dense \item There is a compact $G$-invariant subset $K\subset \Ss^1$ which is infinite and different from $\Ss^1$, such that the orbits of points of $K$ are dense in $K$. \end{enumerate}

\indent In the first case, all finite orbits have the same cardinality.  In the third case, the set $K$ is unique, and it is homeomorphic to a Cantor set. We can call a group $G\subset \Homeo(\Ss^1)$ \textbf{non elementary} if it does not have any finite orbit, and use $L_{G}$ to denote $\Ss^1$ in the second case and  the $G$-invariant compact set $K$ in the third case. If $G\subset \PSL(2,\R)$, then $L_{G}$ is its limit set, so there is no notation conflict.
\subsection{Semi conjugacy} Let us recall a few results of \cite{Gh87} on semi conjugacy.
\begin{defi}  We say that  $\sigma : \Gamma \to \Homeo(\Ss^1)$ is  semi conjugate to  $\tau : \Gamma\to \Homeo(\Ss^1)$ if there is a non constant  map $h:\Ss^1\to\Ss^1$ which is non decreasing of degree one (i.e. that has a non decreasing lift $\tilde h:\R\to \R$ such that $\tilde h(x+1)=\tilde h(x) +1$) such that $h\circ \sigma(\gamma) = \tau(\gamma) \circ h$ for all $\gamma \in \Gamma$. \end{defi}
\indent The advantage of this definition (contrary to the standard definition where $h$ is asked to be continuous) is that semi conjugacy is an equivalence relation.\\
\indent For elementary groups, we have a simple characterisation of semi conjugacy:
\begin{prop} \label{finite_orbit_semi_conjugacy} Let $\rho :\Gamma \to \Homeo(\Ss^1)$ have a finite orbit $E \subset \Ss^1$ with at least two elements. A representation $\tau:\Gamma \to \Homeo(\Ss^1)$ is semi conjugate to $\rho$ if and only if it has a finite orbit $F\subset \Ss^1$ such that there is a cyclic order preserving bijection from $E$ to $F$ which is equivariant under the actions of $\Gamma$. \end{prop}
Consequently, if $\rho$ and $\tau$ are semi conjugate, then $\rho$ is non elementary if and only if $\tau$ is non elementary. If $(M,g)$ is a spatially compact surface that embeds conformally in $\T^2$, then $\rho_1^M(\Isom(M,g))$ is non elementary if and only if $\rho_2^M(\Isom(M,g))$ is non elementary, therefore it is a property of the isometry group.
\begin{defi} Let $(M,g)$ be a spatially compact surface that embeds conformally in $\T^2$. We say that a subgroup $G\subset \Isom(M,g)$ is elementary if $\rho_1^M(G)$ is elementary. \end{defi}

We saw that if there is no finite orbit, then  either there is an invariant Cantor set, or all orbits are dense. However, this distinction is not detectable under semi conjugacy:
\begin{prop} Let $\rho : \Gamma \to \Homeo(\Ss^1)$ be a non elementary representation with an invariant Cantor set $L_{\rho(\Gamma)}$. Then $\rho$ is semi conjugate to a  minimal representation $\hat \rho :\Gamma \to \Homeo(\Ss^1)$. \end{prop}
The proof can be found in \cite{Gh01} and consists in collapsing the circle along the Cantor set $L_{\rho(\Gamma)}$. Our goal for spatially compact surfaces whose isometry group acts non properly is to show that the collapsed representations $\hat \rho_1^M$ and $\hat \rho_2^M$ are topologically conjugate to actions of subgroups of $\PSL_k(2,\R)$. 

\subsection{Rotation number of $h_{\to \uparrow}$} \label{subsec:rotation}
In order to introduce a generalised notion of convergence groups, we will need a homeomorphism that commutes with $\rho_1^M$. If the conformal boundary is acausal, then such a homeomorphism is given by $h_{\to \uparrow} = h_\to \circ h_\uparrow$. Its class under differentiable conjugacy is a conformal invariant.\\ 
\indent The principal invariant under topological conjugacy for circle homeomorphisms is the rotation number. We will see that if the isometry group acts non properly, then there is a restriction on the values that it can take for $h_{\to \uparrow}$.\\
\indent For the $k$-cover of $\dS_2$, we find $h_{\to}=Id$ and $h_{\uparrow}=R_{\frac{1}{k}}$, so $h_{\to \uparrow} = R_{\frac{1}{k}}$. We will see that this is part of a more general property.
\begin{prop} \label{rotation_number} Let $(M,g)$ be a spatially compact surface with an acausal conformal boundary that embeds conformally in $\T^2$. Assume that $\Isom(M,g)$ acts non properly on $M$. Then there is $k\in \N$ such that the rotation number of $h_{\to \uparrow}$ is $\frac{1}{k}$. \end{prop}
\begin{proof} Let $\alpha$ be the rotation number of $h_{\to \uparrow}$, and assume that $\alpha$ is not equal to $\frac{1}{k}$ for some $k\in \N$. Since the cyclic order of the elements of an orbit for $h_{\to \uparrow}$ is the same as for the rotation $R_{\alpha}$,  there are $x\in \Ss^1$ and $n\in \Z$ such that $x<(h_{\to \uparrow})^n(x)<h_{\to \uparrow}(x)<x$. By setting $y=h_{\rightarrow}(x)$, we find that $h_{\downarrow}(y)<(h_{\to \uparrow})^n(h_{\downarrow}(y))<h_{\uparrow}(y)<h_{\downarrow}(y) $, i.e.   $(y,(h_{\to \uparrow})^n(h_{\downarrow}(y)))\in p(M)$. \\
\indent If the graphs of $(h_{\to \uparrow})^n\circ h_{\downarrow}$ and $h_{\downarrow}$  (resp. $h_{\uparrow}$) were to intersect, then $(h_{\to \uparrow})^n$ (resp. $(h_{\to \uparrow})^{n-1}$) would have a fixed point, i.e. its rotation number  $n\alpha$ (resp. $(n-1)\alpha$) is equal to $0$. If $\alpha$ is irrational, this is impossible. If $\alpha$ is rational, then $n$ can be chosen such that $n\alpha =\frac{1}{k}$ for some $k\ge 2$, hence $n\alpha\ne 0$ and $(n-1)\alpha \ne 0$. \\
\indent This implies that $K=Gr((h_{\to \uparrow})^n\circ h_{\downarrow})\subset p(M)$ is a $\Conf(M,g)$-invariant compact set in $M$. Since it is a spacelike circle, we can define a distance on $K$ as the infimum of lengths of spacelike curves joining two points. It is preserved by the isometry group, which shows that the action of $\Isom(M,g)$ on $K$ is equicontinuous. The projection of $K$ onto the first and second coordinates of $\Ss^1\times \Ss^1$ shows that the action of $\Isom(M,g)$ on $K\approx \Ss^1$ is topologically conjugate to $\rho_1^M$ and $\rho_2^M$, hence $\rho_1^M(\Isom(M,g))$ and $\rho_2^M(\Isom(M,g))$ are compact. This implies that $\Isom(M,g)$ acts properly on $p(M)$, hence on $M$.
\end{proof}

\subsection{Properness and compactness}
\begin{lemma} \label{proper_map} Let $(M,g)$ be a spatially compact surface   that embeds conformally in $\T^2$. Assume that the conformal boundary is acausal. Then the maps  $\rho_1^M,\rho_2^M:\Isom(M,g)\to \Homeo(\Ss^1)$ are proper.  \end{lemma}
\begin{proof} We identify $M$ with $p(M)\subset \T^2$. Let $\p_n\to \infty$ in $\Isom(M,g)$. By contradiction, let us assume that the sequence $\rho_1^M(\p_n)$ is equicontinuous. Then, up to a subsequence, it converges to $f_1\in \Homeo(\Ss^1)$. Since $\rho_1^M$ and $\rho_2^M$ are topologically conjugate, we see that there is  $f_2\in \Homeo(\Ss^1)$ such that $\rho_2^M(\p_n)\to f_2$. Let $\p:\T^2\to \T^2$ be defined by $\p(x,y)=(f_1(x),f_2(y))$. Then $\p_n\to \p$ uniformly on $\T^2$, and $\p(M)=M$. Since $\Isom(M,g)$ is a closed subgroup of $\Homeo(M)$ (see \cite{Adams}), we see that $\p \in \Isom(M,g)$ and $\p_n\to \p$ in $\Homeo(M)$ hence in $\Isom(M,g)$, which is absurd because $\p_n\to \infty$. Hence $\rho_1^M(\p_n)\to \infty$. This shows that $\rho_1^M:\Isom(M,g)\to \Homeo(\Ss^1)$ is a proper map, and the same goes for $\rho_2^M$. 
\end{proof} 
As a consequence of this, we see that the groups  $\rho_i^M(\Isom(M,g))$ are closed subgroups of $\Homeo(\Ss^1)$. With the same proof, we would obtain the same for $\Conf(M,g)$ by replacing $\Homeo(\Ss^1)$ with $\Diff(\Ss^1)$ (because $\Conf(M,g)$ is closed in $\Diff(M)$, but not necessarily in $\Homeo(M)$).\\
\indent Let us quote a technical result that we will use.

\begin{lemma} \label{rotation_number_equal} Let $\alpha\in \Ss^1$  and let $f\in \Homeo(\Ss^1)$ have $\alpha$ as its rotation number. There is $x\in \Ss^1$ such that $f(x)=R_{\alpha}(x)$. \end{lemma}
For the proof, see Lemme 4.1.3 in \cite{Herman}.

\begin{prop} \label{proper_compact_group} Let $(M,g)$ be a spatially compact surface with an acausal conformal boundary that embeds conformally in $\T^2$. Let $G\subset \Isom(M,g)$ be a subgroup. The following statements are equivalent: \begin{enumerate} \item $G$ acts properly on $M$ \item $G\subset \Homeo(M)$ is relatively compact \item $\rho_1^M(G)\subset \Homeo(\Ss^1)$ is relatively compact \item $\rho_2^M(G)\subset \Homeo(\Ss^1)$ is relatively compact \end{enumerate} \end{prop}
\begin{proof} $(2)\Rightarrow (1)$ It follows from the definition of a proper action that relatively compact groups always act properly. \\ $(3)\iff (4)$ comes from the fact that $\rho_1^M$ and $\rho_2^M$ are topologically conjugate. \\ $(2)\iff (3)$ is a straightforward consequence of the fact that the map $\rho_1^M:\Isom(M,g)\to \Homeo(\Ss^1)$ is  proper. \\ $(1)\Rightarrow (2)$ We start by considering:  $$W=\{ f\in \Homeo(\Ss^1) \vert \forall x\in \Ss^1~x<f(x)<h_{\rightarrow}\circ h_{\uparrow}(x)\le x\}$$ It is a non empty open set of $\Homeo(\Ss^1)$. For $\alpha$ small enough and positive, the rotation $R_{\alpha}$ is in $W$. Let $\alpha$ be such that $R_{\alpha}\in W$.\\
\indent We note $K=Gr(h_{\downarrow}\circ R_{\alpha})$. If $x\in \Ss^1$, then applying $h_{\downarrow}$ to the inequalities  $x<R_{\alpha}(x)<h_{\rightarrow}\circ h_{\uparrow}(x)\le x$ shows that $(x,h_{\downarrow}\circ R_{\alpha}(x))\in p(M)$, hence $K$ is a compact subset of $M$.\\
\indent Let $\p\in \Isom(M,g)$. Since the rotation number is an invariant under conjugacy, Lemma \ref{rotation_number_equal} implies that there is $x\in \Ss^1$ such that: $$\rho_1^M(\p)\circ R_{\alpha}\circ \rho_1^M(\p)^{-1}(x)=R_{\alpha}(x)$$ Since $h_{\downarrow}\circ \rho_1^M(\p) =\rho_2^M(\p)\circ h_{\downarrow}$, we see that: $$\rho_2^M(\p)\circ (h_{\downarrow}\circ R_{\alpha}) \circ \rho_1^M(\p)^{-1}(x)=h_{\downarrow}\circ R_{\alpha}(x)$$This means that $(x,h_{\downarrow}\circ R_{\alpha}(x))\in  Gr( \rho_2^M(\p)\circ (h_{\downarrow}\circ R_{\alpha}) \circ \rho_1^M(\p)^{-1})= \p(K)$, i.e.  $(x,h_{\downarrow}\circ R_{\alpha}(x))\in K\cap \p(K)$. We have shown that there is a compact set $K\subset M$ such that $K\cap \p(K)\ne \emptyset$ for all $\p\in \Isom(M,g)$. This shows that if $G\subset \Isom(M,g)$ acts properly on $M$, then it is relatively compact. \end{proof}
This result does not hold when $(M,g)$ does not embed conformally in the torus: the isometry group of the flat cylinder is $\R\times \R/\Z$ which is not compact, but it acts properly on the cylinder.\\
\indent We can now prove Theorem \ref{main_theorem_acausal_embeds} in the case where the isometry group acts properly.
\begin{prop} \label{main_acausal_embeds_proper} Let $(M,g)$ be a spatially compact surface that embeds conformally in $\T^2$, with an acausal conformal boundary. If $\Isom(M,g)$ acts properly on $M$, then $\rho_1^M$ is topologically conjugate to a representation in $\mathrm{SO}(2,\R)\subset \PSL(2,\R)$. \end{prop}
\begin{proof} By Lemma \ref{proper_map}, we see that $\rho_1^M(\Isom(M,g))$ is closed in $\Homeo(\Ss^1)$, and Proposition \ref{proper_compact_group} implies that it is compact, therefore topologically conjugate to a subgroup of $\mathrm{SO}(2,\R)$.
\end{proof}

\subsection{The convergence property} Finding a conjugacy between a subgroup of $\Homeo(\Ss^1)$  and a subgroup of $\PSL(2,\R)$, when it exists, is a rather complicated exercise. But there is a characterisation of the existence of such a conjugacy that does not require to find it explicitly.

\begin{defi} A sequence $(f_n)\suite \in \Homeo(\Ss^1)^{\N}$ has the convergence property if there are $a,b\in \Ss^1$ such that, up to a subsequence,  $f_n(x)\to b$ for all $x\ne a$.\\ A group $G \subset \Homeo(\Ss^1)$ is a convergence group if every  sequence in $G$ either satisfies the convergence property or has an equicontinuous subsequence.  \end{defi}
The classical definition of a convergence group also implies the sequence of the inverses $f_n^{-1}$, but it is not necessary in the case of $\Ss^1$. 
\begin{theo}[Gabai \cite{Gabai}, Casson-Jungreis \cite{CJ}] \label{convergence_groups} A convergence group $G \subset \Homeo(\Ss^1)$ is topologically conjugate to a subgroup of $\PSL(2,\R)$. \end{theo}

In the case of finite covers of $\PSL(2,\R)$, we have to generalise the notion of convergence groups to the case where there can be more limit points. The most simple generalisation would be to keep the same definition but to have $k$ possible limit points, as it is the case for subgroups of $\PSL_k(2,\R)$. However, this cannot be enough since the limit points are linked to each other by the rotation of angle $\frac{1}{k}$ in this case. This tells us that we need to add some more data  for a proper generalisation of convergence groups.\\
\indent Let us fix a notation for intervals in $\Ss^1$: set  $\intoo{a}{b}=\{x\in \Ss^1 \vert a<x<b<a\}$ if $a\ne b$ and $\intoo{a}{a}=\Ss^1\setminus \{a\}$.

\begin{defi} Let $k\in \N^*$ and let $h\in \Homeo(\Ss^1)$ have rotation number $\frac{1}{k}$. A sequence $(f_n)\suite\in\Homeo(\Ss^1)^\N$ has the $(h,k)$-convergence property if there are $a,b\in \Ss^1$ such that, up to a subsequence: \begin{itemize} \item $h^k(a)=a$ and $h^k(b)=b$ \item $\forall i\in\{0,\dots,k-1\}~ \forall x\in \intoo{h^i(a)}{h^{i+1}(a)} ~ f_n(x)\to h^i(b)$ \end{itemize} A group  $G\subset \Homeo(\Ss^1)$  is a $(h,k)$-convergence group if all elements of $G$ commute with $h$ and all sequence in $G$ either has the $(h,k)$-convergence property or has an equicontinuous subsequence.
\end{defi}

If $k=1$, then a $(h,k)$-convergence group is a convergence group, hence  topologically Fuchsian.  Of course, the main interest we have in this notion is that it is satisfied by the isometry groups of spatially compact surfaces. If $h$ is topologically conjugate to a rotation, then we immediately obtain an analog of Theorem \ref{convergence_groups}. 
\begin{lemme} \label{convergence_rotation} Let $k\in \N^*$ and let $h$ be topologically conjugate to the rotation of angle $\frac{1}{k}$. If $G\subset \Homeo(\Ss^1)$ is a $(h,k)$-convergence group, then $G$ is topologically conjugate to a subgroup of $\PSL_k(2,\R)$. \end{lemme}

\begin{proof} Let $\pi$ be the projection of the circle $\Ss^1$ onto its quotient by $h$. Since $h$ is topologically conjugate to a rational rotation, it is a finite covering. Since the quotient is homeomorphic to the circle, the image $\pi(G)\subset \Homeo(\Ss^1/h)$ is a subgroup of $\Homeo(\Ss^1)$.\\
\indent Let $(g_n)\suite$ be a sequence in $\pi(G)$ that leaves every compact set. Choose a sequence $(f_n)\suite$ of lifts in $G$. Since $f_n\to \infty$, there are $a,b\in \Ss^1$ such that $f_n(x)\to h^i(b)$ for all $x\in \intoo{h^i(a)}{h^{i+1}(a)}$. Let $x\in \Ss^1 \setminus \pi(a)$. Let $z\in \pi^{-1}(\{x\})$. Then $z$ is not in the orbit of $a$ for $h$, hence $f_n(z)\to h^i(b)$ for some $i\in \{0,\dots,k-1\}$, and $g_n(x)=\pi(f_n(z))\to \pi(h^i(b))=\pi(b)$.\\
\indent We have shown that $\pi(G)\subset \Homeo(\Ss^1)$ is a convergence group, therefore there is $\p\in \Homeo(\Ss^1)$ such that $\p^{-1} \pi(G)\p \subset \PSL(2,\R)$. If $\psi\in \Homeo(\Ss^1)$ is a lift of $\p$, then $\psi^{-1}G\psi \subset \pi^{-1}(\PSL(2,\R))$. Since $h$ is topologically conjugate to a rotation of angle $\frac{1}{k}$, $\pi^{-1}(\PSL(2,\R))$ is topologically conjugate to $\PSL_k(2,\R)$.
\end{proof}

The definition of the $(h,k)$-convergence property can actually simplified by only looking at convergence on one interval.

\begin{lemma} \label{simplify_convergence} Let $h\in \Homeo(\Ss^1)$. Let $G\subset \Homeo(\Ss^1)$ be a group that commutes with $h$. Let  $(f_n)\suite \in G^\N$ . Assume that  there are $a,b\in \Ss^1$ such that  $f_n(x)\to b$ for all $x\in \intoo{a}{h(a)}$. Then there is $k\in \N$ such that the rotation number of $h$ is $\frac{1}{k}$, and $(f_n)\suite$ has the $(h,k)$-convergence property. \end{lemma}
\begin{proof}  Let  $(f_n)\suite \in G^\N$ be such a sequence. Consider $a,b\in \Ss^1$ such that $f_n(x)\to b$ for all $x\in \intoo{a}{h(a)}$.\\
\indent If $x\in \intoo{h^i(a)}{h^{i+1}(a)}$, then $h^{-i}(x) \in \intoo{a}{h(a)}$, which shows that $f_n(x) = h^i\circ f_n\circ h^{-i}(x) \to h^i(b)$. \\
\indent If the rotation number is not equal to some $\frac{1}{k}$, then there is $i\in \N$ such that $h^i(a)\in \intoo{a}{h(a)}$. Let $x\in \intoo{a}{h(a)}$ be close enough to $a$ so that $h^i(x)\in \intoo{a}{h(a)}$. We find that $f_n(x)\to b$, and that $f_n(h^i(x))\to b$. Since $f_n$ commutes with $h$, we also find that $f_n(h^i(x))\to h^i(b)$, hence $h^i(b)=b$. Now let $y\in \intoo{h^{-1}(a)}{a}$ be close enough to $a$ so that $h^i(a)\in \intoo{a}{h(a)}$. We now have $f_n(y)\to h^{-1}(b)$, so $f_n(h^i(b))\to h^{i-1}(b)$ and $h^{i-1}(b)=b$. This implies that $h(b)=b$, so $h$ has a fixed point, therefore its rotation number is $\frac{1}{1}$, which contradicts our assumption.\\
\indent We now have to show that $a$ and $b$ are periodic points of $h$.\\
\indent If $a$ were not a periodic point of $h$, then $h^k(a)$ would belong to an interval $\intoo{h^i(a)}{h^{i+1}(a)}$ for some $i\in \N$. In this case, let $x\in \intoo{a}{h(a)}$ be close enough to $a$ so that $h^k(x) \in \intoo{h^i(a)}{h^{i+1}(a)}$. We then have $f_n(x)\to b$ and $f_n(h^k(x))\to h^i(b)$, so that $h^k(b)=h^i(b)$. This implies that $i=nk$ for some $n\in \N$, and $h^k(b)=b$. Now let $y\in \intoo{h^{-1}(a)}{a}$ be such that $h^k(y)\in  \intoo{h^i(a)}{h^{i+1}(a)}$. We now have $f_n(y)\to h^{-1}(b)$ and $f_n(h^k(y))\to h^i(b)=b$, i.e. $h^{k-1}(b)=b$ which is absurd because all periodic points have the same period.\\
\indent We have shown that $h^k(a)=a$. For any $x\in \intoo{a}{h(a)}$, we have $h^k(x)\in \intoo{a}{h(a)}$, which shows that $f_n(h^k(x))\to b$. Since $f_n(x)\to b$, we also have $f_n(h^k(x))\to h^k(b)$, so finally $h^k(b)=b$.
\end{proof}

 The $(h,k)$-convergence property can be seen as a one dimensional hyperbolic behaviour (there are attracting and repelling points). The key in showing that the isometry groups under study satisfy this property will consist in exhibiting a hyperbolic behaviour for sequences of isometries. In order to find hyperbolic points (attracting in one direction and repelling in the other), we will make use of a result on sequences of affine diffeomorphisms.

\begin{lemma} \label{affine_diffeomorphisms} Let $M$ be a connected manifold equipped with an an affine connection. Let $(f_n)\suite$ be a sequence of affine diffeomorphisms (i.e. that preserve the connection) such that $f_n\to \infty$. If there is a converging sequence $p_n\to p$ in $M$ such that the sequence $f_n(p_n)$ lies in a compact set, then, up to a subsequence, $Df_n(p_n)\to \infty$ or $Df_n(p_n)\to 0$. \end{lemma}

The proof of this result can be found in \cite{Kobayashi}. We will not detail the proof, however the idea of it is quite simple: since affine diffeomorphisms are linearisable via the exponential map, the behaviour of the derivative at one point dictates the behaviour of the diffeomorphism on the whole manifold. In our context, we get a more detailed result. 
\begin{lemma} \label{hyperbolic_lemma} Let $(M,g)$ be a spatially compact surface  that embeds conformally in $\T^2$. Let $\p_n\in\Isom(M,g)$ be such that $\p_n\to \infty$, and let $(x_0,y_0)\in M$ be such that $\p_n(x_0,y_0)\to (x_1,y_1)\in M$. Then, up to a subsequence,  one of the following is satisfied: \begin{itemize} \item $\rho_1^M(\p_n)'(x_0)\to \infty$ and $\rho_2^M(\p_n)'(y_0)\to 0$ \item $\rho_1^M(\p_n)'(x_0)\to 0$ and $\rho_2^M(\p_n)'(y_0)\to \infty$ \end{itemize} \end{lemma}

\begin{proof} The derivative $D\p_n(x_0,y_0)$ is given by the diagonal matrix with coefficients $\rho_1^M(\p_n)'(x_0)$ and $\rho_2^M(\p_n)'(y_0)$. Up to a subsequence, Lemma \ref{affine_diffeomorphisms} gives us four cases.\\
\indent If $D\p_n(x_0,y_0)\to 0$, then either $\rho_1^M(\p_n)'(x_0)\to 0$, either $\rho_2^M(\p_n)'(y_0)\to 0$. Let us write the metric $g=g(x,y)dxdy$. The fact that the maps $\p_n$ are isometries gives us: $$g(\p_n(x_0,y_0))  \rho_1^M(\p_n)'(x_0) \rho_2^M(\p_n)'(y_0) = g(x_0,y_0)$$ Since $g(\p_n(x_0,y_0))\to g(x_1,y_1)\in \R_+^*$, we find that the Jacobian product  $\rho_1^M(\p_n)'(x_0) \rho_2^M(\p_n)'(y_0)$ is bounded in $\R_+^*$, hence the fact that one term converges to $0$ implies that the other tends to $\infty$.\\
\indent In the case where $D\p_n(x_0,y_0)\to \infty$,  one has  either $\rho_1^M(\p_n)'(x_0)\to \infty$ or $\rho_2^M(\p_n)'(y_0)\to \infty$, and the fact that the product $\rho_1^M(\p_n)'(x_0) \rho_2^M(\p_n)'(y_0)$ is bounded in $\R_+^*$ implies that when one term tends to $\infty$, the other converges to $0$.
\end{proof}

This result is exactly the hyperbolic behaviour that we were looking for: we find attraction in one direction and repulsion in the other. A useful fact is that the stable and unstable foliations are simply the lightlike foliations. We are now ready to show that when the conformal boundary is acausal, the isometry groups are convergence groups.

\begin{prop} \label{acausal_convergence} Let $(M,g)$ be a spatially compact surface with an acausal conformal boundary  that embeds in $\T^2$. Assume that the homeomorphisms $h_{\downarrow},h_{\uparrow}$ defining the boundary in $\T^2$ are such that the rotation number of $h_{\rightarrow \uparrow}$ is $ \frac{1}{k}$. Then $\rho_1^M(\Isom(M,g))$ is a $(h_{\rightarrow \uparrow},k)$-convergence group. \end{prop}

\begin{proof} 
 Let $\p_n\in \Isom(M,g)$ be a sequence such that $\rho_1^M(\p_n)\to \infty$ in $\Homeo(\Ss^1)$ (i.e. $\p_n\to \infty$ because of Proposition \ref{proper_compact_group}). We start with $(x_0,y_0)\in M\subset \T^2$, and consider a subsequence such that $\p_n(x_0,y_0)\to (x_1,y_1)\in \overline M\subset \T^2$. \\

\textbf{First case:} Assume that $(x_1,y_1)\in M$.\\
\indent By Lemma \ref{hyperbolic_lemma}, there are two subcases.\\
\emph{First subcase:} $\rho_1^M(\p_n)'(x_0)\to 0$ and $\rho_2^M(\p_n)'(y_0)\to \infty$\\ \indent Let $x\in\Ss^1$ be such that $(x,y_0)\in M$, i.e. $x\in \intoo{h_{\gets}(y_0)}{h_{\rightarrow}(y_0)}$, and consider the geodesic $\gamma$ such that $\gamma(0)=(x_0,y_0)$ and $\gamma(1)=(x,y_0)$. The geodesic $\gamma_n=\p_n\circ \gamma$ has initial data $\gamma_n(0)\to (x_1,y_1)\in M$ and $\gamma'_n(0)\to 0$, hence $\gamma_n$ converges uniformly to a constant geodesic. This implies that $\gamma_n(1)\to (x_1,y_1)$, i.e. $\rho_1^M(\p_n)(x)\to x_1$.\\
\indent Let $a=h_\gets(y_0)$. We have shown that $\rho_1^M(\p_n)(x)\to x_1$ for all $x\in \intoo{a}{h_{\to \uparrow}(a)}$. Lemma \ref{simplify_convergence} implies that the sequence $(\rho_1^M(\p_n))\suite$ has the $(h_{\to \uparrow},k)$-convergence property.
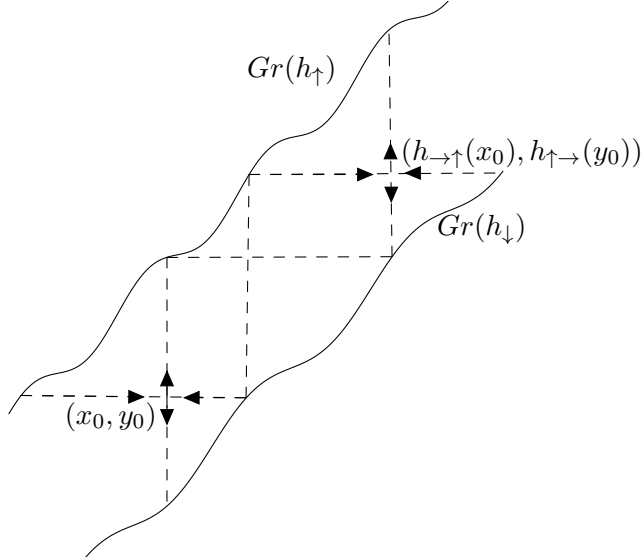
\begin{figure}[h]
\definecolor{uuuuuu}{rgb}{0.27,0.27,0.27}
\definecolor{xdxdff}{rgb}{0.49,0.49,1}
\begin{tikzpicture}[line cap=round,line join=round,>=triangle 45,x=1.0cm,y=1.0cm]
\clip(-3.56,-2.79) rectangle (5.78,4.57);
\draw[smooth,samples=100,domain=-3:3.5] plot(\x,{2+(\x)+0.2*sin((4*(\x))*180/pi)});
\draw[smooth,samples=100,domain=-3:3.5] plot(\x,{0-1+(\x)+0.2*sin((3*(\x)+1)*180/pi)});
\draw [dash pattern=on 4pt off 4pt] (-0.92,1.18)-- (-0.92,-2.12);
\draw [dash pattern=on 4pt off 4pt] (-0.92,1.18)-- (2.04,1.19);
\draw [dash pattern=on 4pt off 4pt] (2.04,1.19)-- (2,4.2);
\draw [dash pattern=on 4pt off 4pt] (-2.84,-0.65)-- (0.12,-0.68);
\draw [dash pattern=on 4pt off 4pt] (0.12,-0.68)-- (0.16,2.28);
\draw [dash pattern=on 4pt off 4pt] (0.16,2.28)-- (3.48,2.3);
\draw [->,dash pattern=on 4pt off 4pt] (0.16,2.28) -- (1.8,2.28);
\draw [->,dash pattern=on 4pt off 4pt] (2.3,2.3) -- (2.2,2.3);
\draw [->,dash pattern=on 4pt off 4pt] (2.02,2.6) -- (2.02,2.7);
\draw [->,dash pattern=on 4pt off 4pt] (2.03,2) -- (2.03,1.9);

\draw (2.5,1.9) node[anchor=north west] {$Gr(h_{\downarrow})$};
\draw (0,4) node[anchor=north west] {$Gr(h_{\uparrow})$};
\draw [->,dash pattern=on 4pt off 4pt] (-2.84,-0.65) -- (-1.21,-0.67);
\draw [->,dash pattern=on 4pt off 4pt] (0.12,-0.68) -- (-0.71,-0.68);
\draw [->,dash pattern=on 4pt off 4pt] (-0.92,-0.67) -- (-0.92,-0.31);
\draw [->,dash pattern=on 4pt off 4pt] (-0.92,-0.67) -- (-0.92,-1.05);
\draw (-2.4,-0.6) node[anchor=north west] {$(x_0,y_0)$};
\draw (2.02,2.9) node[anchor=north west] {$(h_{\rightarrow \uparrow}(x_0),h_{\uparrow \rightarrow}(y_0))$};
\end{tikzpicture}
\caption{The dynamics of the isometry group} \label{fig:convergence_proof}
\end{figure}

\noindent \emph{Second subcase:} $\rho_1^M(\p_n)'(x_0)\to \infty$ and $\rho_2^M(\p_n)'(y_0)\to 0$ \\ \indent In this case, for $x\in \Ss^1$ such that $(x,y_0)\in M$, the geodesic joining $(x_0,y_0)$ to $(x,y_0)$ is now dilated by the sequence $\p_n$, which shows that $\rho_1^M(\p_n)(x)$ converges to $h_{\gets}(y_1)$ for $x\in \intfo{h_{\gets}(y_0)}{x_0}$ and to $h_{\rightarrow}(y_1)$ for $x\in \intof{x_0}{h_{\rightarrow}(y_0)}$.\\
\indent If $x\in \intfo{h_{\rightarrow}(y_0)}{h_{\rightarrow \uparrow}(x_0)}$, then $h_{\rightarrow \uparrow}^{-1}(x)\in \intfo{h_{\gets}(y_0)}{x_0}$ which shows that $\rho_1^M(\p_n)(x)\to h_{\rightarrow \uparrow}(h_{\gets}(y_1))=h_{\rightarrow}(y_1)$.\\
\indent We have shown that $\rho_1^M(\p_n)(x)\to h_{\rightarrow}(y_0)$ for all $x\in \intoo{x_0}{h_{\rightarrow \uparrow}(x_0)}$. Lemma \ref{simplify_convergence} implies that $(\rho_1^M(\p_n))\suite$ has the $(h_{\to \uparrow},k)$-convergence property.\\
\indent We now know that if  $(x_1,y_1)\in M$, then $(\rho_1^M(\p_n))\suite$ has the  $(h_{\rightarrow \uparrow},k)$-convergence property.\\

\textbf{Second case:} Assume that $(x_1,y_1)\notin M$.\\
\indent If there is  $x\in \Ss^1$ such that $(x,y_0)\in M$ and $\rho_1^M(\p_n)(x)\to x'$  with $(x',y_1)\in M$, then the first case shows that $(\rho_1^M(\p_n))\suite$ has the $(h_{\rightarrow \uparrow},k)$-convergence property. Therefore we can assume that there is no such $x$. In this case, the only limit points of $\rho_1^M(\p_n)(x)$ are $h_{\gets}(y_1)$ and $h_{\rightarrow}(y_1)$. We now have three subcases.\\
\emph{First subcase:} $\rho_1^M(\p_n)(x)\to h_{\gets}(y_1)$ for all $x\in\intoo{h_{\gets}(y_0)}{h_{\rightarrow}(y_0)}$.\\ \indent  Since $h_\to(y_0)=h_{\to \uparrow}(h_\gets(y_0))$, Lemma \ref{simplify_convergence} implies that $(\rho_1^M(\p_n))\suite$ has the $(h_{\to \uparrow},k)$-convergence property.\\
\emph{Second subcase:} $\rho_1^M(\p_n)(x)\to h_{\rightarrow}(y_1)$ for all $x\in \intoo{h_{\gets}(y_0)}{h_{\rightarrow}(y_0)}$.\\ \indent  The argument is the same as in the previous case since we only change the limit.\\
\emph{Third subcase:} The two limits are possible.\\ \indent If $\rho_1^M(\p_n)(u)\to h_{\gets}(y_1)$ for some $u\in \intoo{h_{\gets}(y_0)}{h_{\rightarrow}(y_0)}$, then $\rho_1^M(\p_n)(x)\to h_{\gets}(y_1)$ for all $x\in \intof{h_{\gets}(y_0)}{u}$. Similarly, if $\rho_1^M(\p_n)(v)\to h_{\rightarrow}(y_1)$ for some $v\in \intoo{h_{\gets}(y_0)}{h_{\rightarrow}(y_0)}$, then $\rho_1^M(\p_n)(x)\to h_{\rightarrow}(y_1)$ for all $x\in \intfo{v}{h_{\rightarrow}(y_0)}$. This implies that there is a point $z\in  \intoo{h_{\gets}(y_0)}{h_{\rightarrow}(y_0)}$ such that $\rho_1^M(\p_n)(x)\to h_{\gets}(y_1)$ for all $x\in \intoo{h_{\gets}(y_0)}{z}$ and $\rho_1^M(\p_n)(x)\to h_{\rightarrow}(y_1)$ for all $x\in \intoo{z}{h_{\rightarrow}(y_0)}$.\\
\indent If $x\in \intoo{h_{\rightarrow}(y_0)}{h_{\rightarrow \uparrow}(z)}$, then $h_{\rightarrow \uparrow}^{-1}(x)\in \intoo{h_{\gets}(y_0)}{z}$, which shows that $\rho_1^M(\p_n)(x)\to h_{\rightarrow \uparrow}(h_{\gets}(y_1))=h_{\rightarrow}(y_1)$. Once again, Lemma \ref{simplify_convergence} implies that $(\rho_1^M(\p_n))\suite$ has the $(h_{\to \uparrow},k)$-convergence property.\\
\end{proof}

Note that the strategy consisting in separating the cases depending on the possible limit points can also be found in Theorem 2.5 of \cite{Ba96}.\\
\indent We do not know if the $(h,k)$-convergence property implies topological conjugacy with a subgroup of $\PSL_k(2,\R)$, but it will still be a crucial tool in our proof for isometry groups of spatially compact surfaces. We already obtain Theorem \ref{main_theorem_acausal_embeds} in a special case (when $k=1$).
 
\begin{coro} \label{connected} Let $(M,g)$ be a spatially compact surface  with an acausal conformal boundary that embeds conformally in $\T^2$. Assume that the boundary of $p(M)\subset \T^2$ is connected. Then $\rho_1^M(\Isom(M,g))$ is topologically conjugate to a subgroup of $\PSL(2,\R)$. \end{coro}
\begin{proof} If the boundary of $p(M)$ is connected, then $h_{\rightarrow \uparrow}$ has a fixed point, i.e. $k=1$. We showed that $\rho_1^M(\Isom(M,g))$ is a $(h_{\rightarrow \uparrow},1)$-convergence group, i.e. a convergence group, hence it is topologically conjugate to a subgroup of $\PSL(2,\R)$. \end{proof}

\section{Reducing the problem to open sets of $\T^2$} \label{sec:reducing_problem} Before we go further, we will see that under the assumption that the isometry group acts non properly, isometries are completely described by the representations $\rho_1^M$ and $\rho_2^M$, i.e. they are faithful. In the study of surfaces that do not embed conformally in $\T^2$, we will make use of the second conformal model (embedding in the flat cylinder, defined in \ref{subsec:flat_cylinder}).

\begin{lemme} \label{cylinder_proper} Let $(M,g)$ be a spatially compact surface that is conformal to the flat cylinder. Then $\Isom(M,g)$ acts properly on $M$.
\end{lemme}
\begin{proof} If the action were not proper, we could find a sequence $f_n$ in $\Isom(M,g)$ that leaves every compact set, and converging sequences $a_n\to a$ and $b_n\to b$ in $M$ such that $f_n(a_n)=b_n$. We can lift everything to $\widetilde M=\R^2$: we choose lifts $\widetilde a_n,\widetilde a, \widetilde b_n, \widetilde b$ such that $\widetilde a_n\to \widetilde a$ and $\widetilde b_n\to \widetilde b$, and lifts $\widetilde f_n$ of $f_n$ such that $\widetilde f_n(\widetilde a_n)=\widetilde b_n$. We write the lifts $\widetilde f_n(x,y)=(\alpha_n(x),\beta_n(y))$ and $\widetilde a_n=( x_n, y_n), \widetilde b= (u,v)$. Since $\widetilde f_n\to \infty$, Lemma \ref{affine_diffeomorphisms} implies that we have $\alpha'_n(x_n)\to \infty$ or $\beta'_n(y_n)\to \infty$. In the first case, we have $\beta'_n(y_n)\to 0$ (just as in Lemma \ref{hyperbolic_lemma}).\\ \indent Consider the geodesic going from $(x_n,y_n)$ to $(x_n,y_n+1)$. The image of this geodesic by $\widetilde f_n$ converges towards a constant geodesic (the initial vector shrinks). This shows that $\widetilde f_n(x_n,y_n+1)\to (u,v)$, i.e. $\beta_n(y_n+1)\to v$, which is incompatible with $\beta_n(y_n+1)=\beta_n(y_n)+1\to v+1$.\\ \indent The same reasoning applied to the geodesic joining $(x_n,y_n)$ and $(x_n+1,y_n)$ treats the other case. Hence $\Isom(M,g)$ acts properly on $M$.
\end{proof}

Associating this and the following proposition, we obtain the first parts of Theorem \ref{main_theorem} and Theorem \ref{main_theorem_acausal}.

\begin{prop} \label{faithful} Let $(M,g)$ be a spatially compact surface that is not conformal to the flat cylinder. Then $\rho_1^M$ and $\rho_2^M$ are semi conjugate to each other, and their restrictions to $\Isom(M,g)$ are faithful. Moreover, if the conformal boundary is acausal, then they are topologically conjugate. \end{prop}
\begin{proof} 
Since $(M,g)$ is not conformally equivalent to the flat cylinder, at least one of $\tilde h_{\downarrow},\tilde h_{\uparrow}$ is not a constant and provides a semi conjugacy between $\rho_1^M$ and $\rho_2^M$.\\
\indent Let $f\in \ker(\rho_1^M) $, and let $\widetilde f$ be a lift to $\widetilde M$, written $\widetilde f(x,y)=(\alpha(x),\beta(y))$. Since $\rho_1^M(f)=Id$, we find that $\alpha(x)-x\in \Z$ for all $x\in \R$. The continuity of $f$ implies that there is $n\in \Z$ such that $\alpha(x)=x+n$ for all $x\in \R$. Consider $A=T^{-n}\circ \widetilde f$ (where $T(x,y)=(x+1,y+1)\in \Isom(\widetilde M, \widetilde g)$). It is also a lift of $f$, that can be written $A(x,y)=(x,\gamma(y))$ where $\gamma$ is semi conjugate to the identity via $\tilde h_{\uparrow}$ or $\tilde h_{\downarrow}$, hence $\gamma$ has fixed points. If $\gamma(y)=y$, then we choose $x\in \R$ such that $(x,y)\in \tilde M$. Since $A(x,y)=(x,y)$ and $A$ is an isometry, the Jacobian at $(x,y)$ is equal to $1$, i.e. $\gamma'(y)=1$. This implies that $A$ is an isometry with a fixed point $(x,y)\in M$ where the differential is the identity, therefore $A=Id$,  and $f=Id$, i.e. $\rho_1^M$ is injective. The same goes for $\rho_2^M$.
\end{proof}

This result gives a similarity with the case where $M$ embeds in the torus. We will now see that from the isometry group point of view, there is no difference.

\begin{prop} \label{open_set_embeds} Let $(M,g)$ be a spatially compact surface  that does not embed conformally in $\T^2$, such that the action of $\Isom(M,g)$ on $M$ is non proper. Then there is an open set $U\subset M$ such that: \begin{itemize} \item $U$ is  invariant under $\Conf(M,g)$\item $(U,g_{/U})$ is spatially compact \item The restriction map $r :\Isom(M,g) \to \Isom(U,g_{/U})$ is injective \item  $\rho_1^U\circ  r=\rho_1^M$ and $\rho_2^U\circ r=\rho_2^M$    \item  $(U,g_{/U})$ embeds conformally in $(\T^2,dxdy)$ 
 \item If the conformal boundary of $(M,g)$ is acausal, then the same goes for $(U,g_{/U})$. In this case, the image of $U$ in $\T^2$ has a connected boundary. 
\end{itemize} \end{prop}

This implies that all the results in the case where $(M,g)$ embeds in the torus   still apply when $(M,g)$ does not embed in the torus, provided that the isometry group acts non properly. Therefore,  Theorem \ref{main_theorem_embeds} implies Theorem \ref{main_theorem} and Theorem \ref{main_theorem_acausal_embeds} implies Theorem \ref{main_theorem_acausal}. The proof will make use of two intermediate results.

\begin{lemma} \label{limit_boundary} Let $(M,g)$ be a spatially compact surface. Consider its universal cover  $\widetilde M=\{(x,y)\in \R^2 \vert \tilde h_{\downarrow}(x)<y<\tilde h_{\uparrow}(x)\}$. Let  $(x_0,y_0)\in M$, let $(f_n)\suite$ be a sequence in $\Isom(M,g)$ such that $f_n\to \infty$ and let $\widetilde f_n$ be a sequence of lifts to $\widetilde M$. Assume that $\widetilde f_n(x_0,y_0)\to (x_1,y_1)\in \R^2$. If $(x_0,y_0+1)\in \widetilde M$, or $(x_0,y_0-1)\in \widetilde M$, then $(x_1,y_1)\notin \widetilde M$. \end{lemma}

\begin{proof} Write $\widetilde f_n(x,y)=(\alpha_n(x),\beta_n(y))$. Let us assume that $(x_1,y_1)\in M$, and that $(x_0,y_0+1)\in \widetilde M$ (the case where $(x_0,y_0-1)\in \widetilde M$ is identical). In that case, $\alpha'_n(x_0)\to \infty$ or $\beta'_n(y_0)\to \infty$. Assume that $\alpha'_n(x_0)\to \infty$. Then $\beta'_n(y_0)\to 0$. Let $\gamma$ be the geodesic with starting point $\gamma(0)=(x_0,y_0)$ and end point $\gamma(1)=(x_0,y_0+1)$. The image $\eta_n=\widetilde f_n\circ \gamma$ has initial value $\eta_n(0)\to (x_1,y_1)$ and $\eta'(0)=\beta'_n(y)\gamma'(0)\to 0$, hence $\eta$ converges uniformly to a constant geodesic, and $\widetilde f_n\circ \gamma(1)\to (x_1,y_1)$, which is absurd because $\beta_n(y_0+1)=\beta_n(y_0)+1\to y_1+1\ne y_1$. In the case where $\beta'_n(y_0)\to \infty$, then $\alpha'_n(x_0)\to 0$ and we use the same reasoning with the geodesic joining $(x_0,y_0)$ to $(x_0-1,y_0)\in \widetilde M$. This shows that $(x_1,y_1)\notin \widetilde M$.

\end{proof}

We have already seen in Lemma \ref{cylinder_proper} that the conformal geometry can give obstructions to the non properness of the action of the isometry group. The consequence of the following result can be explained in terms of intersection of lightlike geodesics: the fact that $\tilde h_{\uparrow}(x)-\tilde h_{\downarrow}(x)\le 2$ for all $x\in \R$ means that there are no lightlike geodesics with three or more intersection points, and the fact that there is some $x_0\in \R$ such that $\tilde h_{\uparrow}(x_0)-\tilde h_{\downarrow}(x_0)\le 1$ means there are some lightlike geodesics with exactly one intersection point.

\begin{lemma} \label{lemma_open_set_embeds} Let $(M,g)$ be a spatially compact surface. Consider its universal cover $\widetilde M =\{ (x,y)\in \R^2 \vert \tilde h_{\downarrow}(x)<y<\tilde h_{\uparrow}(x)\}$. If $\Isom(M,g)$ acts non properly on $M$, then: \begin{itemize} \item $\tilde h_{\downarrow}\ne -\infty$ and $\tilde h_{\uparrow}\ne +\infty$ \item $\forall x\in \R~ \tilde h_{\uparrow}(x)-\tilde h_{\downarrow}(x)\le 2$ \end{itemize}  If the conformal boundary of $(M,g)$ is acausal, then there is  $x_0\in \R$ such that $\tilde h_{\uparrow}(x_0)-\tilde h_{\downarrow}(x_0)\le 1$ 
\end{lemma}

\begin{proof} Note that if $\tilde h_{\uparrow}=\infty$ (resp. $\tilde h_{\downarrow}=-\infty$), then $(x,y+1)\in \widetilde M$ (resp. $(x,y-1)\in \widetilde M$) for all $(x,y)\in \widetilde M$. In this case, Lemma \ref{limit_boundary} shows that the action of $\Isom(M,g)$ on $M$ is proper.\\
\indent Let us assume that there is $x_0\in \R$ such that $\tilde h_{\uparrow}(x_0)-\tilde h_{\downarrow}(x_0)>2$ Let $y_0\in \R$ be such that $\tilde h_{\downarrow}(x_0)<y_0-1<y_0+1<\tilde h_{\uparrow}(x_0)$. Let $f_n\to \infty$ in $\Isom(M,g)$ and let $\widetilde f_n$ be a sequence of lifts to $\widetilde M$ such that the sequence $\widetilde f_n(x_0,y_0)$ lies in a compact set of $\R^2$. Up to a subsequence, we consider that $\widetilde f_n(x_0,y_0)\to (x_1,y_1)$. By Lemma \ref{limit_boundary}, we know that $(x_1,y_1)\notin \widetilde M$, hence $y_1=\tilde h_{\uparrow}(x_1)$ or $y_1=\tilde h_{\downarrow}(x_1)$. In the case where $y_1=\tilde h_{\uparrow}(x_1)$, we have  $\beta_n(y_0)\le \beta_n(y_0+1)\le \tilde h_{\uparrow}(\alpha_n(x_0))$, which shows that $\beta_n(y_0+1)\to y_1$. This is impossible because $\beta_n(y_0+1)=\beta_n(y_0)+1\to y_1+1$. Similarly, if $y_1=\tilde h_{\downarrow}(x_1)$, we find $\beta_n(y_1-1)\to y_1$, which is absurd.\\
\indent For the third statement, notice that if $\tilde h_{\uparrow}(x)-\tilde h_{\downarrow}(x)>1$ for all $x\in \R$, then  $L=\{(x,\tilde h_{\downarrow}(x)+1)\vert x\in \R\}\subset \widetilde M$  is homeomorphic to the real line and preserved by the conformal group. Since it is spacelike (when the conformal boundary is acausal), the isometry group preserves a distance on $L$ that is bi-lipschitz to the euclidian distance, and it acts properly on $M$.
\end{proof}

\begin{proof}[Proof of Proposition \ref{open_set_embeds}]  By Lemma \ref{lemma_open_set_embeds}, we know that $\tilde h_{\uparrow}$ and $\tilde h_{\downarrow}$ are not constants. Let $\tilde h(x)$ denote $\tilde h_{\uparrow}(x)$ if $\tilde h_{\uparrow}(x)\le \tilde h_{\downarrow}(x)+1$ and $\tilde h_{\downarrow}(x)+1$ if $\tilde h_{\uparrow}(x)<\tilde h_{\downarrow}(x)+1$. It is a non decreasing map such that $\tilde h(x+1)=\tilde h(x)+1$, and it commutes with lifts of conformal diffeomorphisms of $M$. Consider $V=\{ (x,y)\in \R^2 \vert \tilde h_{\downarrow}(x)<y<\tilde h(x)\}\subset \widetilde M$ and let $U$ be its image in $M$. It is an open set invariant under $\Conf(M)$. Since it contains some Cauchy surfaces of $M$ and it is causally convex (i.e. an inextensible causal curve in $U$ is the  intersection of $U$ and an inextensible causal curve in $M$), it is spatially compact. Since $\tilde h(x)-\tilde h_{\downarrow}(x)\le 1$, we see that lightlike geodesics in $U$ cannot have several intersections, therefore $U$ embeds conformally in $\T^2$. \\ \indent  The lightlike geodesics of $U$ are the lightlike geodesic of $M$, it follows  immediately that $\rho_1^U\circ r = \rho_1^M$ and $\rho_2^U\circ r = \rho_2^M$.  In particular, we find that $\ker(r) \subset \ker(\rho_1^M)=\{Id\}$, so $r$ is injective. \\ \indent If $M$ does not embed in the torus and the conformal boundary is acausal, then there is $x_0\in \R$ such that $\tilde h_{\uparrow}(x_0)-\tilde h_{\downarrow}(x_0)>1$, hence $\tilde h(x_0)-\tilde h_{\downarrow}(x_0)=1$. This shows that the image in the torus has a connected boundary.

\end{proof}

\begin{coro} Let $(M,g)$ be a spatially compact surface with an acausal conformal boundary  that does not embed conformally in $\T^2$, such that the action of $\Isom(M,g)$ on $M$ is non proper. Then $\rho_1^M(\Isom(M,g))$ is a convergence group, hence is topologically conjugate to a subgroup of $\PSL(2,\R)$. \end{coro}
\begin{proof} Let $U\subset M$ be the open set given by Proposition \ref{open_set_embeds}. By Proposition \ref{connected}, we see that $\rho_1^M(\Isom(M,g))=\rho_1^U\circ r(\Isom(M,g))$ is topologically conjugate to a subgroup of $\PSL(2,\R)$.\end{proof}

\section{Proof of Theorem \ref{main_theorem}} \label{sec:proof_main}
\subsection{Elementary groups} We wish to prove Theorem \ref{main_theorem} for elementary groups. Let us start with the stabilizer of a point.

\begin{lemma} \label{stabilizer_semi_conjugacy} Let $(M,g)$ be a spatially compact surface that embeds conformally in $\T^2$. Assume that $G\subset \Isom(M,g)$ is  such that $\rho_1^M(G)$ fixes a point $x_0\in \Ss^1$. There is a faithful representation $\rho : G \to \PSL(2,\R)$ that is semi conjugate to the restriction of $\rho_1^M$ to $G$.  
\end{lemma}

\begin{proof} Since $\rho_1^M(G)$ fixes $x_0$, the representation $\rho_2^M(G)$ fixes $y_0=h_\uparrow(x_0)$. This implies that $G$ fixes the horizontal geodesic $(\Ss^1\times \{y_0\})\cap M$.\\
\indent The parametrisation of this geodesic gives a representation $\rho : G\to \Aff(\R)\subset \PSL(2,\R)$ and a diffeomorphism $\p : \intoo{h_\gets(y_0)}{h_\to(y_0)} \to \R$ such that $\p \circ \rho_1^M = \rho \circ \p$ (we get the real line $\R$ and not just an open interval because $G$ acts non trivially on this geodesic).\\
\indent Let us show that $\rho$ is faithful. If $\p \in \ker(\rho)$, then $\p$ fixes all points on the horizontal geodesic $(\Ss^1\times \{y_0\})\cap M$. If $(x,y_0)\in M$, we then get $\p(x,y_0) =(x,y_0)$ and $\rho_1^M(\p)'(x)=1$, therefore $\rho_2^M(\p)'(y_0)=1$ (because the Jacobian is equal to $1$) and $\p$ is an isometry having a fixed point where its derivative is the identity, therefore $\p=Id$.\\
\indent Let $\psi : \Ss^1\to \Ss^1=\R\cup \{\infty\}$ be defined by $\psi =\p$ on $\intoo{h_\gets(y_0)}{h_\to(y_0)}$ and $\psi\equiv \infty$ on $\intff{h_\to(y_0)}{h_\gets(y_0)}$. It provides a semi conjugacy between $\rho_1^M(G)$ and the action of $\rho$ on the circle.
\end{proof}

\begin{prop} \label{semi_conjugacy_elementary} Let $(M,g)$ be a spatially compact surface that embeds conformally in $\T^2$. Assume that $G\subset  \Isom(M,g)$ is  elementary. There are $k\in \N$ and a faithful representation $\rho : G \to \PSL_k(2,\R)$ that is semi conjugate to the restriction of $\rho_1^M$ to $G$. \end{prop}

\begin{proof} Up to considering the closure $\overline G$, we can assume that $G$ is closed. Let $L_1 \subset \Ss^1$ be a finite orbit of $\rho_1^M(G)$.\\
\indent Lemma \ref{stabilizer_semi_conjugacy} treats the case where $\sharp L_1 =1$, therefore we can assume that $\sharp L_1\ge 2$. Let $n=\sharp L_1$, and consider $L_1=\{x_{\overline 1},x_{\overline 2},\dots,x_{\overline n}\}$ where the indices are taken in $\Z/n\Z$.\\
\indent Since $\rho_1^M$ preserves the cyclic ordering, there is a  morphism $\sigma : G\to \Z/n\Z$ such that $\rho_1^M(\p)(x_i)=x_{i+\sigma(\p)}$ for all $\p\in G$ and $i\in \Z/n\Z$.  Since $G$ acts transitively on $L_1$, we necessarily have $\sigma(G)=\Z/n\Z$.\\
\indent Let $\p_1\in G$ be such that $\sigma(\p_1)=\overline 1$, and let   $H=\ker \p=\Stab(x_1)$.   \\
\indent  If $H=\{Id\}$, then $G=\langle \p_1\rangle$ and it is semi conjugate any element of $\PSL_n(2,\R)$  having the same rotation number as $\p_1$. Such an element can be chosen to be of finite order (if $\p_1^n=Id$) or not, so that the corresponding subgroup of $\PSL_n(2,\R)$ is isomorphic to $G$.\\
\indent We now assume that $H$ is non trivial. The proof of Lemma \ref{stabilizer_semi_conjugacy} shows that the group $H$ is isomorphic to a closed subgroup of $\Aff(\R)$ hence isomorphic to either $\Z$, $\R$ or $\Aff(\R)$.  If it is isomorphic to $\Aff(\R)$, then its orbits are dense in $M$, so $(M,g)$ has constant curvature and $\Isom(M,g)$ is differentially conjugate to a subgroup of $\PSL_k(2,\R)$. Indeed, the developing map $D:\tilde M \to N$ (where $N$ is either $\R^{1,1}$ or $\tilde \dS_2$) is the map $\tilde p$ defined in  \ref{subsec:flat_cylinder}. This implies that $D$ is injective, so $M$ is the quotient of an open set of $\R^{1,1}$ or $\tilde \dS_2$, and the representations of the isometries in $\Diff(\Ss^1)$ are either in some $\PSL_k(2,\R)$ or in $\mathrm{SO}(2,\R)\subset \PSL(2,\R)$. \\
\indent We now assume that $H$ is either isomorphic to $\Z$ or to $\R$.  The group $K$ generated by $\p_1$ acts on $H$ by conjugacy, which shows that $G$ is a semi direct product $H\rtimes K$.\\
\indent If the action of $K$ on $H$ is trivial, i.e. if $G\approx H\times K$, then it is isomorphic and semi conjugate to a subgroup of $\PSL_n(2,\R)$, taking either the group generated by an element of the center of $\PSL(2,\R)$ and the corresponding subgroup of $\Aff(\R)$ (seen as the stabilizer of a point in $\PSL_n(2,\R)$), when $K\approx \Z/n\Z$, either the group generated by a parabolic element of $\PSL_n(2,\R)$ with same rotation number $\frac{1}{n}$ and the corresponding subgroup of $\Aff(\R)$, when $K\approx \Z$.\\
\indent We now assume that $K$ acts non trivially on $H$.  Since $\p_1^n\in H$ and $H$ is abelian, this implies that the action of $K$ on $H$ is done by a finite order automorphism of $K$. There is only one such non trivial element (the map $x\mapsto -x$ in $\Z$ or $\R$), and it is of order two. This implies that there is $k\in \N$ such that $n=2k$. One can realise such a group in $\PSL_k(2,\R)$ by considering the group generated by a hyperbolic element and an elliptic element that exchanges its fixed points.

\end{proof}

\subsection{Non elementary groups}

\subsubsection{The collapsed actions} Recall that if $\rho : \Gamma \to \Homeo(\Ss^1)$ is a non elementary representation with an invariant Cantor set $L_{\rho(\Gamma)}$, then one can construct a minimal representation $\hat \rho: \Gamma \to \Homeo(\Ss^1)$ by considering a continuous non decreasing map of degree one $\pi :\Ss^1\to \Ss^1$ obtained by collapsing the connected components of $\Ss^1\setminus L_{\rho(\Gamma)}$ to points. We then define $\hat \rho$ so that it satisfies $\hat \rho \circ \pi = \pi \circ \rho$.\\
\indent Let $(M,g)$ be a spatially compact surface that embeds conformally in $\T^2$, and assume that $\Isom(M,g)$ is non elementary.  We denote by $\pi_1, \pi_2:\Ss^1\to \Ss^1$ and $\hat \rho_1^M, \hat \rho_2^M$ the maps and representations obtained for the representations $\rho_1^M,\rho_2^M$. Note that $\hat \rho_1^M$ and $\hat \rho_2^M$ are representations of $\Isom(M,g)$, i.e. they do not necessarily extend to $\Conf(M,g)$, since the conformal group does not necessarily preserve the minimal sets $L_{\rho_1^M(\Isom(M,g))}$ and $L_{\rho_2^M(\Isom(M,g))}$.\\
\indent In general, the fact that $\rho$ is faithful does not imply that $\hat \rho$ is. However, it is the case for representations associated to spatially compact surfaces.

\begin{prop} \label{collapsed_faithful} Let $(M,g)$ be a spatially compact surface that embeds conformally in $\T^2$. If  $\Isom(M,g)$ is non elementary, then the collapsed representations $\hat \rho_1^M$ and $\hat \rho_2^M$ are faithful. \end{prop}
\begin{proof} Let $\p \in \Isom(M,g)$ be such that $\hat \rho_1^M(\p)=Id$. If $x\in L_{\rho_1^M(\Isom(M,g))}$, then there are two possibilities. Either $x$ bounds an interval $I$ of $\Ss^1\setminus  L_{\rho_1^M(\Isom(M,g))}$, in which case the fact that $\hat \rho_1^M(\p)(\hat x)= \hat x$ implies that $\rho_1^M(\p)$ is equal to $x$ or to the other endpoint of $I$; either $x$ can be approached in both directions by elements of $ L_{\rho_1^M(\Isom(M,g))}$, in which case $\hat y=\hat x$, therefore $\rho_1^M(x)=x$.\\
\indent If $x\in  L_{\rho_1^M(\Isom(M,g))}$ is of the first kind, then it can be approached by a sequence $x_n$ in $ L_{\rho_1^M(\Isom(M,g))}$ such that $\rho_1^M(\p)(x_n)\to x$, therefore $\rho_1^M(\p)(x)=x$ for all $x\in  L_{\rho_1^M(\Isom(M,g))}$. This implies that $\rho_1^M(\p)'(x)=1$ for all $x\in  L_{\rho_1^M(\Isom(M,g))}$.\\
\indent The fixed points of $\rho_2^M(\p)$ contain the closure of  $h_\to(L_{\rho_1^M(\Isom(M,g))})$ which is $\rho_2^M$-invariant, which implies that $\rho_2^M(\p)(y)=y$ and $\rho_2^M(\p)'(y)=1$ for all $y\in L_{\rho_2^M(\Isom(M,g))}$. Taking $(x,y)\in M\cap  (L_{\rho_1^M(\Isom(M,g))}\times  L_{\rho_2^M(\Isom(M,g))})$, we have a fixed point of the isometry $\p$ where the derivative is the identity, hence $\p=Id$. This shows that $\hat \rho_1^M$ is faithful, and so is $\hat \rho_2^M$.
\end{proof}

We will use the fact that the map $h_{\to \uparrow}$ gives a homeomorphism that commutes with the collapsed representation $\hat \rho_1^M$.

\begin{prop} Let $(M,g)$ be a spatially compact surface that embeds conformally in $\T^2$. Assume that $\Isom(M,g)$ is non elementary. There is $\hat h_{\to \uparrow} \in \Homeo(\Ss^1)$ such that $\hat h_{\to \uparrow} \circ \pi_1 = \pi_1 \circ h_{\to \uparrow}$. It commutes with $\hat \rho_1^M$.
\end{prop}
\begin{proof} Let $\hat x = \pi_1(x)\in \Ss^1$. We wish to show that $\pi_1\circ h_{\to \uparrow}(x)$ only depends on $\hat x$. It is enough to show that if $I=\intoo{a}{b}$ is a connected component of $\Ss^1\setminus L_{\rho_1^M(\Isom(M,g))}$, then $h_{\to \uparrow}(\overline I)$ is included in the closure of a connected component of $\Ss^1\setminus L_{\rho_1^M(\Isom(M,g))}$. If it were not the case, there would be $y\in L_{\rho_1^M(\Isom(M,g))}$ such that $h_{\to \uparrow}(a)<y<h_{\to \uparrow}(b)\le h_{\to \uparrow}(a)$.\\
\indent Since $\overline{h_{\to \uparrow}(L_{\rho_1^M(\Isom(M,g))})}$ is closed an invariant under $\rho_1^M$, it contains $L_{\rho_1^M(\Isom(M,g))}$, so there is $z\in h_{\to \uparrow}(L_{\rho_1^M(\Isom(M,g))})$ such that $h_{\to \uparrow}(a)<z<h_{\to \uparrow}(b)\le h_{\to \uparrow}(a)$. If $z=h_{\to \uparrow}(u)$ with $u\in L_{\rho_1^M(\Isom(M,g))}$, then we find $u\in I$, which is absurd, and shows that $\hat h_{\to \uparrow}$ is well defined.\\
\indent Notice that $\hat \rho_1^M\circ \hat h_{\to \uparrow}\circ \pi_1= \pi_1 \circ \rho_1^M \circ h_{\to \uparrow} = \pi_1  \circ h_{\to \uparrow} \circ \rho_1^M =  \hat h_{\to \uparrow} \circ \hat \rho_1^M\circ \pi_1$. Since $\pi_1$ is onto, this shows that $\hat h_{\to \uparrow}$ commutes with $\hat \rho_1^M$.\\
\indent  Since $\hat h_{\to \uparrow}$  is non decreasing of degree one, the union of the open intervals  where it is constant is invariant under $\hat \rho_1^M$, and therefore is empty, so $\hat h_{\to \uparrow}$ is injective. Similarly, it is onto, and continuous, so $\hat h_{\to \uparrow} \in \Homeo(\Ss^1)$.
\end{proof}

We can also define $\hat h_\to = \pi_1\circ h_\to$ and $\hat h_\gets =\pi_1\circ h_\gets$. We will use the fact that they are linked by $\hat h_{\to \uparrow}$.

\begin{prop} \label{collapsed_commutation} $\hat h_{\to \uparrow} \circ \hat h_\gets = \hat h_\to$ \end{prop}

\begin{lemma} Let $I\subset \Ss^1$ be an interval such that $h_{\gets}$ only takes a finite number of values on $I$. Then $\mathring I\cap L_{\rho_1^M(\Isom(M,g))} = \emptyset$.
\end{lemma}
\begin{proof} Assume that $\mathring I\cap L_{\rho_2^M(\Isom(M,g))}$ is non empty. Let $x\in L_{\rho_1^M(\Isom(M,g))}$. Since the action of $\Isom(M,g)$ on $L_{\rho_1^M(\Isom(M,g))}$ is minimal, the orbit of $x$ meets $\mathring I\cap L_{\rho_1^M(\Isom(M,g))}$, so $x$ has a neighbourhood on which $h_{\gets}$ only takes a finite number of values. Since $L_{\rho_1^M(\Isom(M,g))}$ is compact, this implies that $h_{\gets}(L_{\rho_1^M(\Isom(M,g))})$ is a finite set invariant under $\rho_2^M$, which is absurd.
\end{proof}
 
\begin{proof}[Proof of Proposition \ref{collapsed_commutation}]  First, we see that $\hat h_{\to \uparrow}\circ \hat h_\gets = \pi_1 \circ h_{\to \uparrow} \circ h_\gets = \hat h_\to \circ h_{\uparrow \gets}$. We wish to show that $\hat h_\to \circ h_{\uparrow \gets} = \hat h_\to$.\\
\indent Let $x\in \Ss^1$ be such that $h_{\uparrow \gets}(x)\ne x$. This means that $x$ lies on  an interval where $h_{\gets }$ is constant. If $\hat h_\to \circ h_{\uparrow \gets}(x) \ne \hat h_\to(x)$, then the interval $\intoo{h_\to \circ h_{\uparrow \gets} (x)}{h_\to(x)}$ intersects $L_{\rho_1^M(\Isom(M,g))}$, but $h_\gets$ is constant on this interval. According to the previous lemma, this is a contradiction. Therefore $\hat h_{\to \uparrow}\circ \hat h_\gets = \hat h_\to$.
\end{proof}

\subsubsection{Convergence property for the collapsed actions}
\begin{lemma} \label{collapsed_convergence} Let $(M,g)$ be a spatially compact surface that embeds conformally in $\T^2$. Assume that $\Isom(M,g)$ is non elementary. Then either $\hat \rho_1^M(\Isom(M,g))$ is topologically conjugate to a subgroup of $\mathrm{SO}(2,\R)$, either the rotation number of $\hat h_{\rightarrow \uparrow}$ is equal to $\frac{1}{k}$ for somme $k\in \N$, and $\hat \rho_1^M (\Isom(M,g))$ is a $(\hat h_{\rightarrow \uparrow},k)$-convergence group.
\end{lemma}

\begin{proof} If $\hat \rho_1^M(\Isom(M,g))$ is compact, then it is topologically conjugate to a subgroup of $\mathrm{SO}(2,\R)$, in particular it has the convergence property. We can therefore assume that $\hat \rho_1^M(\Isom(M,g))$ is non compact, i.e. that sequences $\p_n$ such that $\hat \rho_1^M(\p_n)\to \infty$ exist.\\
\indent Let  $(x_0,y_0)\in M$. Consider a sequence $(\p_n)\suite \in \Isom(M,g)^\N$ such that $\hat \rho_1^M(\p_n)$ has no equicontinuous subsequence. Since $\hat \rho_1^M$ is continuous, this implies that $\p_n\to \infty$.\\
\indent Up to a subsequence, we can assume that there are $x_1,y_1\in \Ss^1$ such that $\rho_1^M(\p_n)(x_0)\to x_1$ and $\rho_2^M(\p_n)(y_0)\to y_1$.\\


\indent \textbf{First case:} Assume that $(x_1,y_1)\in M$.\\
\indent By Lemma \ref{hyperbolic_lemma}, there are two subcases.\\
\emph{First subcase:} $\rho_1^M(\p_n)'(x_0)\to 0$ and $\rho_2^M(\p_n)'(y_0)\to \infty$\\ \indent Just as in the proof of Proposition \ref{acausal_convergence}, the horizontal geodesic passing through $(x_0,y_0)$ is shrunk to the point $(x_1,y_1)$, i.e. $\rho_1^M(\p_n)(x)\to x_1$ for all $x\in \intoo{h_\gets(y_0)}{h_\to(y_0)}$. Let $z=h_\gets(y_0)$. \\
\indent On the collapsed circle, we see that $\hat \rho_1^M(\p_n)(\hat x)\to \hat x_1$ for all $\hat x\in \intoo{\hat z}{\hat h_{\to \uparrow}(\hat z)}$, and Lemma \ref{simplify_convergence} implies that  the rotation number of $\hat h_{\to \uparrow}$ is some $\frac{1}{k}$ and that $(\hat \rho_1^M(\p_n))\suite$ has the $(\hat h_{\to \uparrow},k)$-convergence property.\\
\emph{Second subcase:} $\rho_1^M(\p_n)'(x_0)\to \infty$ and $\rho_2^M(\p_n)'(y_0)\to 0$\\  \indent The horizontal geodesic passing through $(x_0,y_0)$ is now dilated. If $\hat x\in \intoo{\hat x_0}{\hat h_\to(y_0)}$, then $\hat  \rho_1^M(\p_n)(x)\to \hat h_\to(y_1)$. If $\hat x \in \intoo{\hat h_\to(y_0)}{\hat h_{\to \uparrow}(\hat x_0)}$, then $\hat h_{\to \uparrow}^{-1}(\hat x)\in \intoo{\hat h_\gets(y_0)}{\hat x_0}$, so $\hat \rho_1^M(\p_n)(\hat h_{\to \uparrow}^{-1}(\hat x))\to \hat h_\gets(y_1)$, and $\hat \rho_1^M(\p_n)(\hat x)\to \hat h_{\to \uparrow}(\hat h_\gets(y_1)) = \hat h_\to(y_1)$.\\
\indent We have shown that $\hat \rho_1^M(\p_n)(\hat x)\to \hat h_\to(y_1)$ for all $\hat x\in \intoo{\hat x_0}{\hat h_{\to \uparrow}(\hat x_0)}\setminus \{\hat h_\to(y_0)\}$. By monotonicity, we have convergence on the whole interval, so Lemma \ref{simplify_convergence} can once again be applied.\\

\textbf{Second case:} Assume that $(x_1,y_1)\notin M$.\\
Just as in Proposition \ref{acausal_convergence}, we can assume that there is no $x\in \Ss^1$ such that $(x,y_0)\in M$ and such that the sequence $\rho_1^M(\p_n)(x)$ has a limit point $z\in \Ss^1$ satisfying $(z,y_1)\in M$. This implies that for all $\hat x\in  \intoo{\hat h_\gets(y_0)}{\hat h_\to(y_0)}$, the only limit points of the sequence $\hat \rho_1^M(\p_n)(\hat x)$ are $\hat h_\gets(y_1)$ and $\hat h_\to(y_1)$. Up to a subsequence, we have three possibilities.\\
\emph{First subcase:} $\hat \rho_1^M(\p_n)(\hat x) \to \hat h_\gets(y_1)$ for all $\hat x\in  \intoo{\hat h_\gets(y_0)}{\hat h_\to(y_0)}$.\\ \indent  Since $\hat h_\to(y_0) = \hat h_{\to \uparrow}(\hat h_\gets(y_0))$, Lemma \ref{simplify_convergence} implies that   the rotation number of $\hat h_{\to \uparrow}$ is some $\frac{1}{k}$ and that $(\hat \rho_1^M(\p_n))\suite$ has the $(\hat h_{\to \uparrow},k)$-convergence property..\\
\emph{Second subcase:} $\hat \rho_1^M(\p_n)(\hat x) \to \hat h_\to(y_1)$ for all $\hat x\in  \intoo{\hat h_\gets(y_0)}{\hat h_\to(y_0)}$.\\ \indent The reasoning is exactly the same as in the previous case.\\
\emph{Third subcase:}  The two limits are possible. \\ \indent As in Proposition \ref{acausal_convergence}, there is $z\in \intoo{h_\gets(y_0)}{h_\to(y_0)}$ such that $\hat \rho_1^M(\p_n)(\hat x)\to \hat h_\gets(y_1)$ for all $\hat x\in \intoo{\hat h_\gets(y_0)}{\hat z}$ and $\hat \rho_1^M(\p_n)(\hat x)\to \hat h_\to(y_1)$ for all $\hat x\in \intoo{\hat z}{\hat h_\to(y_0)}$. This implies that $\hat \rho_1^M(\p_n)(\hat x) \to \hat h_\to(y_1)$ for all $\hat x\in \intoo{\hat z}{\hat h_{\to\uparrow}(\hat z)}$, and we once again conclude with Lemma \ref{simplify_convergence}.
\end{proof}

We now have all the ingredients for the proof of Theorem \ref{main_theorem_embeds}, which implies Theorem \ref{main_theorem} because of Proposition \ref{open_set_embeds}.

\begin{proof}[Proof of Theorem \ref{main_theorem_embeds}] Let $(M,g)$ be a spatially compact surface  that embeds conformally in $\T^2$.\\
\indent Proposition \ref{faithful} implies that $\rho_1^M$ and $\rho_2^M$ are semi conjugate, and that their restrictions to $\Isom(M,g)$ are faithful.\\
\indent If $\Isom(M,g)$ is  elementary, then Proposition \ref{semi_conjugacy_elementary} states that there is a faithful representation $\rho:\Isom(M,g)\to \PSL_k(2,\R)$ for some $k\in \N$ that is semi conjugate to $\rho_1^M$.\\ \indent If $\Isom(M,g)$ is non elementary, then  Lemma \ref{collapsed_convergence} assures that either $\hat \rho_1^M$ is topologically conjugate to a representation in $\mathrm{SO}(2,\R)\subset \PSL(2,\R)$, either the rotation number of $\hat h_{\to \uparrow}$ is equal to some $\frac{1}{k}$ and that $\hat \rho_1^M(\Isom(M,g))$ is a $(\hat h_{\to \uparrow},k)$-convergence group. Since the periodic points of $\hat h_{\uparrow}$ form a non empty closed set invariant under $\hat \rho_1^M(\Isom(M,g))$, it is equal to $\Ss^1$ and $\hat h_{\to \uparrow}^k=Id$. Since $\hat h_{\to \uparrow}^k=Id$, we see that $\hat h_{\to \uparrow}$  is topologically conjugate to the rotation of angle $\frac{1}{k}$, and Lemma \ref{convergence_rotation} states that  $\hat \rho_1^M(\Isom(M,g))$ is topologically conjugate to a subgroup of $\PSL_k(2,\R)$. Since the collapsed action $\hat \rho_1^M$ is faithful (Proposition \ref{collapsed_faithful}) and semi conjugate to $\rho_1^M$, we have finished the proof of Theorem \ref{main_theorem}.
\end{proof}

\section{Conjugacy for elementary groups} \label{sec:elementary_acausal}

The goal of this section is to prove  Theorem \ref{main_theorem_acausal} in the case of elementary groups. If $(M,g)$ is a spatially compact surface and $G\subset \Isom(M,g)$ is a subgroup, we will identify $G$ and $\rho_1^M(G)$. 
\subsection{Classification of elements and finite invariant sets} Rather than looking at finite orbits, it will be more practical to consider certain finite invariant sets on which the group may not act transitively.
\begin{lemma} \label{finite_invariant_set} Let $k\in \N$, and let $h\in \Homeo(\Ss^1)$ have rotation number $\frac{1}{k}$. Let $G\subset \Homeo(\Ss^1)$ be an elementary $(h,k)$-convergence group. Then $G$ has a finite invariant set $L_G$ satisfying one of the following properties: \begin{enumerate} \item $L_G$ has more than $2k$ points \item $L_G=\{x_0,y_0,h(x_0),h(y_0),\dots,h^{k-1}(x_0),h^{k-1}(y_0)\}$ where $x_0$ and $y_0$ are periodic points for $h$ such that $x_0<y_0<h(x_0)$
\item $L_G=\{x_0,h(x_0),\dots,h^{k-1}(x_0)\}$ where $h^k(x_0)=x_0$. \end{enumerate} \end{lemma}

\begin{proof} Let $E\subset \Ss^1$ be a finite invariant set. If $\# E >2k$, then we are in the first case. If $h(E)\ne E$, then $G$ also preserves $h(E)$. If $E$ has elements that are not periodic for $h$, then $E$ preserves $E\cup h(E) \cup \cdots \cup h^n(E)$ for all $n$. For $n$ large enough, it has more than $2k$ elements. Therefore we can assume that all elements of $E$ are periodic for $h$, and by adding the iterates under $h$ we can assume that $\# E$ is a multiple of $k$. If it is $3k$ or more, then we are in the first case. If $\# E=2k$, then the second condition is satisfied. Finally, if $\# E=k$, then the third condition is satisfied.
\end{proof}

Applying this to the group generated by one element, we obtain a classification of elements similar to the case of $\PSL(2,\R)$. If $G$ is a $(h,k)$-convergence group, and $f\in G\setminus \{Id\}$, then we say that $f$ is \begin{itemize} \item Hyperbolic if $f$ has exactly $2k$ periodic points. \item Parabolic if $f$ has exactly $k$ periodic points. \item Elliptic if it is not hyperbolic or parabolic. \end{itemize}

Note that if $\gamma$ is elliptic, then the group generated by $\gamma$ is not always elementary (think of irrational rotations).

\subsection{The elliptic case: $\# L_G>2k$} \begin{lemma} \label{elliptic} Let $k\in \N$ let $h\in \Homeo(\Ss^1)$ have rotation number $\frac{1}{k}$. Let $G\subset \Homeo(\Ss^1)$ be a closed elementary $(h,k)$-convergence group. If $G$ has a finite invariant set $L_G\subset \Ss^1$ with more than $2k$ elements, then $G$ is compact. \end{lemma}

\begin{proof} If $\# L_G>2k$, then we can find three points $x_1,x_2,x_3\in L_G$ such that $x_1<x_2<x_3<h(x_1)$. Let us assume that there is a sequence $(f_n)$ in $G$ such that $f_n\to \infty$. Since the images of $x_1,x_2,x_3$ under the $f_n$ belong to the finite set $L_G$, up to a subsequence there are $y_1,y_2,y_3\in L_G$ such that $f_n(x_i)=y_i$ for $i=1,2,3$. This shows that the sequence $(f_n)$ does not satisfy the $(h,k)$-convergence property, which is impossible because $G$ is a $(h,k)$-convergence group. Therefore there is no sequence $(f_n)$ in $G$ such that $f_n\to \infty$, i.e. $G$ is relatively compact. Since it is closed, it is compact.
\end{proof}

\subsection{The hyperbolic case: $\# L_G=2k$} \begin{lemma} \label{hyperbolic} Let $(M,g)$ be a spatially compact surface with an acausal conformal boundary that embeds in $\T^2$ and let  $G\subset \Isom(M,g)$ be an elementary closed subgroup. Let $h_{\uparrow},h_{\downarrow}$ be the homeomorphisms that define the boundary in $\T^2$, and assume that the rotation number of $h_{\rightarrow \uparrow}$ is $\frac{1}{k}$. If $G$ has a finite invariant set $L_G\subset \Ss^1$ such that $L_G=\bigcup_{i=0}^{k-1} \{h^i_{\rightarrow \uparrow}(x_0),h^i_{\rightarrow \uparrow}(y_0)\}$ where $h_{\rightarrow \uparrow}^k(x_0)=x_0$ and $h_{\rightarrow \uparrow}^k(y_0)=y_0$, then $G$ is topologically conjugate to a subgroup of $\PSL_k(2,\R)$. \end{lemma}

\begin{proof} We note $H=\{f\in G \vert f(x_0)=x_0 \textrm{ and } f(y_0)=y_0\}$. \\

\textbf{First step:} Show that $H\approx \R$ or $H\approx \Z$.\\
Let $\Phi: H \to \R$ be defined by $\Phi(f)=\textrm{Log}(f'(x_0))$. If $\Phi(f)=0$, then $f$ fixes the point $(x_0,h_\downarrow(y_0))\in M$ and its derivative at this point is the identity, hence $f=Id$. We showed that $\Phi$ is an injective homomorphism, and $H$ is isomorphic to a closed subgroup of $\R$.\\

\begin{figure}[h]
\begin{tikzpicture}[line cap=round,line join=round,>=triangle 45,x=1.0cm,y=1.0cm]
\clip(-5.06,-1.67) rectangle (7,4.26);
\draw [shift={(0.46,1.02)}] plot[domain=-0.91:3.14,variable=\t]({1*2.63*cos(\t r)+0*2.63*sin(\t r)},{0*2.63*cos(\t r)+1*2.63*sin(\t r)});
\draw [shift={(0.46,1.02)},dash pattern=on 2pt off 2pt]  plot[domain=3.14:5.37,variable=\t]({1*2.63*cos(\t r)+0*2.63*sin(\t r)},{0*2.63*cos(\t r)+1*2.63*sin(\t r)});
\draw (-0.39,4.1) node[anchor=north west] {$x_0$};
\draw (0.99,4) node[anchor=north west] {$y_0$};
\draw (2.75,2.93) node[anchor=north west] {$h_{\rightarrow \uparrow}(x_0)$};
\draw (3.17,2.2) node[anchor=north west] {$h_{\rightarrow \uparrow}(y_0)$};
\draw (3.13,0.48) node[anchor=north west] {$h_{\rightarrow \uparrow}^2(x_0)$};
\draw (2.75,-0.22) node[anchor=north west] {$h_{\rightarrow \uparrow}^2(y_0)$};
\draw (-3.7,2.26) node[anchor=north west] {$h_{\rightarrow \uparrow}^{k-1}(x_0)$};
\draw (-3.35,3.07) node[anchor=north west] {$h_{\rightarrow \uparrow}^{k-1}(y_0)$};
\draw (2.82,2.32)-- (2.82,2.18);
\draw (2.82,2.18)-- (2.68,2.25);
\draw (3.01,1.04)-- (3.08,1.24);
\draw (3.08,1.24)-- (3.18,1.05);
\draw (2.7,-0.16)-- (2.71,-0.33);
\draw (2.71,-0.33)-- (2.88,-0.21);
\draw (1.91,3.12)-- (1.79,3.29);
\draw (1.79,3.29)-- (1.99,3.27);
\draw (-0.84,3.22)-- (-1.04,3.18);
\draw (-1.04,3.18)-- (-0.93,3.36);
\draw (0.36,3.54)-- (0.5,3.65);
\draw (0.5,3.65)-- (0.32,3.73);
\draw (2.66,2.25)-- (2.82,2.17);
\draw (2.82,2.17)-- (2.81,2.38);
\draw (-1.82,2.12)-- (-1.84,2.29);
\draw (-1.84,2.29)-- (-2,2.17);
\begin{scriptsize}
\fill [color=black] (-0.22,3.56) circle (1.5pt);
\fill [color=black] (1,3.59) circle (1.5pt);
\fill [color=black] (2.59,2.56) circle (1.5pt);
\fill [color=black] (2.93,1.93) circle (1.5pt);
\fill [color=black] (2.9,0.04) circle (1.5pt);
\fill [color=black] (2.59,-0.53) circle (1.5pt);
\fill [color=black] (-1.67,2.56) circle (1.5pt);
\fill [color=black] (-1.99,1.98) circle (1.5pt);
\end{scriptsize}
\end{tikzpicture}
\caption{Dynamics of an element of $H$} \label{fig:dynamics_H}
\end{figure}
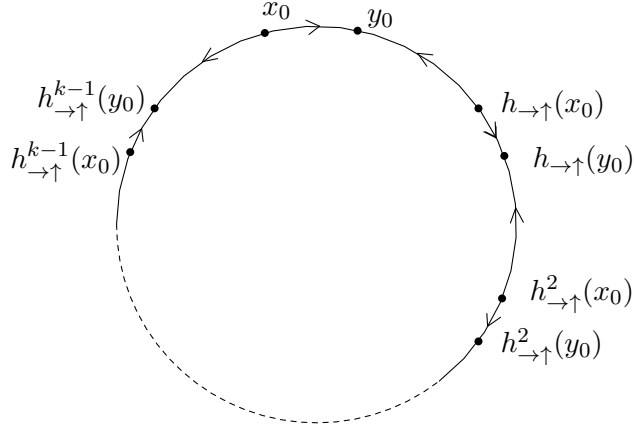

\textbf{Second step:} Find a conjugacy for $H$\\
\emph{First case:} When $H\approx \R$. We set $H= \{ f_t \vert t\in \R\}$. We start by choosing $\widetilde x_0,\widetilde y_0\in \Ss^1$ such that $\widetilde x_0<\widetilde y_0<\widetilde x_0 +\frac{1}{k}$ and $\gamma_t$ the (unique) one parameter subgroup of $\PSL_k(2,\R)$ such that $\gamma_t(\widetilde x_0)=\widetilde x_0$, $\gamma_t(\widetilde y_0)=\widetilde y_0$ and $\gamma'_1(\widetilde x_0)=f'_1(x_0)$ (i.e. $\gamma_t$ has the same dynamics as $f_t$, see Figure \ref{fig:dynamics_H}). We also choose $z_i\in  \intoo{h_{\rightarrow \uparrow}^i(x_0)}{h_{\rightarrow \uparrow}^i(y_0)}$, $ z'_i \in \intoo{h_{\rightarrow \uparrow}^i(y_0)}{h_{\rightarrow \uparrow}^{i+1}(x_0)}$, $\widetilde z_i\in \intoo{\widetilde x_0 + \frac{i}{k}}{\widetilde y_0+\frac{i}{k}}$, and $\widetilde z'_i\in \intoo{ \widetilde y_0+\frac{i}{k}}{ \widetilde x_0 + \frac{i+1}{k}}$. \\
\indent If $x\in \intoo{h_{\rightarrow \uparrow}^i(x_0)}{h_{\rightarrow \uparrow}^i(y_0)}$, then we consider $t_x\in \R$ such that $x=f_{t_x}(z_i)$ and we set $\p(x)=\gamma_{t_x}(\widetilde z_i)$. If $x\in \intoo{h_{\rightarrow \uparrow}^i(y_0)}{ h_{\rightarrow \uparrow}^{i+1}(x_0)}$, then we consider $t_x\in \R$ such that $x=f_{t_x}(z'_i)$ and we set $\p(x)=\gamma_{t_x}(\widetilde z'_i)$. Since $t_{f_s(x)}=t+s$, we see that $\p \circ f_t=\gamma_t \circ \p$, and $\p$ is a homeomorphism.\\
\emph{Second case:} When $H\approx \Z$. We set $H=\{ f_1^n\vert n\in \Z\}$. We start by choosing $\widetilde x_0,\widetilde y_0\in \Ss^1$ such that $\widetilde x_0<\widetilde y_0<\widetilde x_0 +\frac{1}{k}$ and $\gamma_1\in \PSL_k(2,\R)$ a hyperbolic element such that $\gamma_1(\widetilde x_0)=\widetilde x_0$, $\gamma_1(\widetilde y_0)=\widetilde y_0$ and $\gamma'_1(\widetilde x_0)=f'_1(x_0)$. We also choose $z_i\in  \intoo{h_{\rightarrow \uparrow}^i(x_0)}{h_{\rightarrow \uparrow}^i(y_0)}$, $ z'_i \in \intoo{h_{\rightarrow \uparrow}^i(y_0)}{ h_{\rightarrow \uparrow}^{i+1}(x_0)}$, $\widetilde z_i\in \intoo{\widetilde x_0 + \frac{i}{k}}{ \widetilde y_0+\frac{i}{k}}$, and $\widetilde z'_i\in \intoo{ \widetilde y_0+\frac{i}{k}}{ \widetilde x_0 + \frac{i+1}{k}}$. Finally, we can choose arbitrarily the restrictions  $\p: \intff{z_i}{f_1(z_i)}\to \intff{\widetilde z_i}{\gamma_1(\widetilde z_i)}$ and $\p: \intff{z'_i}{f_1(z'_i)}\to \intff{\widetilde z'_i}{\gamma_1(\widetilde z'_i)}$.\\
\indent If $x\in \intoo{h_{\rightarrow \uparrow}^i(x_0)}{h_{\rightarrow \uparrow}^i(y_0)}$, then we consider the unique $n_x\in \Z$ such that $f_1^{n_x}(x)\in \intfo{z_i}{f_1(z_i)}$, and we set $\p(x)=\gamma_1^{-n_x}\circ \p\circ f_1^{n_x}(x)$. Similarly, if $x\in \intoo{h_{\rightarrow \uparrow}^i(y_0)}{ h_{\rightarrow \uparrow}^{i+1}(x_0)}$, then we consider the unique $n_x\in \Z$ such that $f_1^{n_x}(x)\in \intfo{z'_i}{f_1(z'_i)}$, and we set $\p(x)=\gamma_1^{-n_x}\circ \p\circ f_1^{n_x}(x)$. This construction gives a homeomorphism $\p$ such that $\p \circ f_1=\gamma_1\circ \p$.\\

\textbf{Third step:} Show that if $f\in G$,  then  $f(x_0)=h_{\rightarrow \uparrow}^i(x_0)$ for some $i\in \N$.\\
Let $f_1\in H$ have $x_0$ as an attracting fixed point. We first wish to show that $f_1$ and $f$ commute. For this, we see that $f^{-1}\circ f_1\circ f(x_0)=f(x_0)$, i.e. $f^{-1}\circ f_1\circ f\in H$. If $H=<f_1>\approx \Z$, then there is $n\in \N$ such that $f^{-1}\circ f_1\circ f = f_1^n$. But $f^k\in H$, therefore $f^k$ commutes with $f_1$, which shows that $f_1=f_1^n$, hence $n=1$. If $H\approx \R$, then there is $t$ such that $f^{-1}\circ f_1\circ f = f_t$ and $f^k\in H$ shows $f_1=f_t$, hence $t=1$.\\
Now that we know that $f$ and $f_1$ commute, we choose $x\in \Ss^1$ sufficiently close to $x_0$ so that $f_1^n(x)\to x_0$ and $f(x)\ne h_{\rightarrow \uparrow}^j(y_0)$ for all $j$. Then $f_1^n(x)\to h_{\rightarrow \uparrow}^i(x_0)$ for some $i$, and $f(x_0)=\lim f(f_1^n(x)) = \lim f_1^n(f(x)) =h_{\rightarrow \uparrow}^i(x_0)$.\\

\textbf{Fourth step:} Find a conjugacy for $H$ and one element of $G\setminus H$.\\
Note that we showed in the third step that $H$ is contained in the center of $G$. Let $f\in G\setminus H$, and let $\gamma \in \PSL_k(2,\R)$ have the same rotation number as $f$ and have $\widetilde x_0, \widetilde y_0$ as periodic points. If $x_0$ is attracting (resp. repelling) for $f^k$, then we choose $\gamma$ such that $\widetilde x_0$ is attracting (resp. repelling) for $\gamma^k$ (i.e. $f$ and $\gamma$ have the same dynamics). \\
\emph{First case:} $H\approx \R$. We wish to choose  $\p$ such that $\gamma \circ \p(z_j)=\p \circ f(z_j)$ and $\gamma \circ \p(z'_j)=\p \circ f(z'_j)$. The left side of the equation is $\gamma \circ \p(z_j)=\gamma(\widetilde z_j)$. On the right side, we have $\p\circ f(z_j)$. Let $t_j\in \R$ be such that $f(z_j)=f_{t_j}(z_{i+j})$. We see that $\p\circ f(z_j)=\gamma_{t_j}(\widetilde z_{i+j})$. Hence the first equation holds if and only if $\widetilde z_{i+j} = \gamma_{-t_j} \circ \gamma(\widetilde z_j)$. This shows that we can fix the family $(z_j)$ arbitrarily, then choose $\widetilde z_j$ for one $j$ in each class in $\Z/i\Z$, and set $\widetilde z_{i+j}= \gamma_{-t_j} \circ \gamma(\widetilde z_j)$. The same goes for the $z'_j$ and $\widetilde z'_j$.\\
\indent For $x\in \intoo{h_{\rightarrow \uparrow}^j(x_0)}{h_{\rightarrow \uparrow}^j(y_0)}$, let $t\in \R$ be such that $x=f_t(z_j)$.
\begin{eqnarray*} \p\circ f(x) = \p\circ f\circ f_t(z_j) &=& \p \circ f_t \circ f(z_j) ~~~~~~~ f_t\in Z(G)\\ &=& \gamma_t \circ \p \circ f(z_j)   ~~~~~~~\p \textrm{ conjugates } f_t \textrm{ and } \gamma_t \\ &=& \gamma_t \circ \gamma \circ \p(z_j) ~~~~~~~\textrm{Choice of } \widetilde z_{i+j}\\ &=& \gamma\circ \gamma_t \circ \p(z_j) ~~~~~~~ \gamma_t \textrm{ commutes with } \gamma \\&=& \gamma \circ \p(x) ~~~~~~~~~~~~~~ \textrm{Definition of } \p \end{eqnarray*} The same calculations hold for $x\in \intoo{h_{\rightarrow \uparrow}^j(y_0)}{h_{\rightarrow \uparrow}^{j+1}(x_0)}$. This shows that $\p \circ f=\gamma \circ \p$.\\
\emph{Second case:} $H\approx \Z$. We wish to choose $\p$ such that $\gamma\circ \p(u)=\p\circ f(u)$ for $u\in \intfo{z_j}{f_1(z_j)}\cup\intfo{z'_j}{f_1(z'_j)}$. For $u\in \intfo{z_j}{f_1(z_j)}$, we consider $n_u\in \Z$ such that $f(u)=f_1^{n_u}(x_u)$ where $x_u\in \intfo{z_{i+j}}{f_1(z_{i+j})}$. We get $\p\circ f(u)=\gamma \circ \p (u)$ if and only if $\p(x_u)=\gamma_1^{-n_u}\circ \gamma\circ \p(u)$. Hence we only need to choose $\p$ on $\intfo{z_j}{f_1(z_j)}$ for one $j$ in each class modulo $i$, and set $\p$ on $\intfo{z_{i+j}}{f_1(z_{i+j})}$ by the formula $\p(x_u)=\gamma_1^{-n_u}\circ \gamma\circ \p(u)$. We do the same on $\intfo{z'_j}{f_1(z'_j)}$.\\
\indent Finally, for $x\in \intoo{h_{\rightarrow \uparrow}^j(x_0)}{h_{\rightarrow \uparrow}^j(y_0)}$, we consider $n\in \Z$ such that $x=f_1^n(u)$ where $u\in \intfo{z_j}{f_1(z_j)}$. \begin{eqnarray*} \p\circ f(x)=\p\circ f\circ f_1^n(u) &=& \p \circ f_1^n\circ f(u)\\ &=& \gamma_1^n \circ \p \circ f(u)\\ &=& \gamma_1^n \circ \gamma \circ \p(u) \\&=& \gamma \circ \gamma_1^n \circ \p(u)\\&=& \gamma \circ \p(x)\end{eqnarray*} The same calculations holds for $x\in \intoo{h_{\rightarrow \uparrow}^j(y_0)}{h_{\rightarrow \uparrow}^{j+1}(x_0)}$. This shows that $\p \circ f=\gamma \circ \p$.\\

\textbf{Fifth step:} Show that it provides a conjugacy for $G$.\\
Let $F\subset \Z/k\Z$ be the set of classes of numbers $i\in \Z$ such that there is $f\in G$ satisfying $f(x_0)=h_{\rightarrow \uparrow}^i(x_0)$. It is a subgroup of $\Z/k\Z$, hence there is $n\in \{0,\dots,k-1\}$ such that $F=n\Z/k\Z$. Let $f\in G$ be such that $f(x_0)=h_{\rightarrow \uparrow}^n(x_0)$, and let $\p\in \Homeo(\Ss^1)$ be a conjugacy for $H$ and for $f$.\\ 
If $u$ is another element of $G\setminus H$, then there is $i\in \Z$ such that  $u\circ f^{-i} \in H$, hence $\p\circ u\circ \p^{-1} = (\p\circ (u\circ f^{-i}) \circ \p^{-1}) \circ (\p\circ f^i\circ \p^{-1}) \in \PSL_k(2,\R)$.
\end{proof}

\subsection{The parabolic case: $\# L_G=k$} \begin{lemma} \label{parabolic} Let $(M,g)$ be a spatially compact surface with an acausal conformal boundary that embeds in $\T^2$ and let  $G\subset \Isom(M,g)$ be an elementary closed subgroup. Let $h_{\uparrow},h_{\downarrow}$ be the homeomorphisms that define the boundary in $\T^2$, and assume that the rotation number of $h_{\rightarrow \uparrow}=h_{\rightarrow}\circ h_{\uparrow}$ is $\frac{1}{k}$. If $G$ has a finite invariant set  $L_G\subset \Ss^1$ such that $L_G=\{x_0,h_{\rightarrow \uparrow}(x_0),\dots,h_{\rightarrow \uparrow}^{k-1}(x_0)\}$ where $h_{\rightarrow \uparrow}^k(x_0)=x_0$, then $G$ is topologically conjugate to a subgroup of $\PSL_k(2,\R)$. \end{lemma}

\begin{proof} Let $H=\{ f\in G \vert f(x_0)=x_0\}$. \\

\textbf{First step:} Assume that $H$ contains a hyperbolic element $f$.\\
Up to replacing $f$ by $f^{-1}$, we can assume that $x_0$ is an attracting point for $f$. If all elements of $H$ fix the other fixed point $y_0$ of $f$, then we have an invariant set with $2k$ elements and Lemma \ref{hyperbolic} shows that $G$ is topologically conjugate to a subgroup of $\PSL_k(2,\R)$. Hence we can assume that there is $u \in H$ that does not fix $y_0$.\\
\emph{First case:} Assume that $u$ is parabolic. Up to replacing $u$ by $u^{-1}$, assume that $u(y_0)\in \intoo{y_0}{x_0}$. We consider the sequence $u_n=f^{-n}\circ u \circ f^n$. We have $u_n(x_0)=x_0$ and $u_n(y_0)=f^{-n}(u(y_0)) \to y_0$. If $x\in \intoo{x_0}{y_0}$, then $x_0<f^n(x)<u(f^n(x))<y_0<x_0$ gives $x_0<f^{-n}(f^n(x))=x<u_n(x)<y_0$ (see Figure \ref{fig:parabolic}), hence the sequence $u_n(x)$ does not have $x_0$ as a limit point. If $y\in \intoo{y_0}{h_{\rightarrow \uparrow}(x_0)}$, then we find $y_0<y<u_n(y)<h_{\rightarrow \uparrow}(x_0)<y_0$ and the sequence $u_n(y)$ does not have $y_0$ as a limit point. This shows that the sequence $u_n$ does not have the $(h_{\rightarrow \uparrow},k)$-convergence property, hence it is equicontinuous. This implies that $G$ is not discrete. Since it is a Lie group, it has dimension at least one, and there is a one parameter subgroup of parabolic elements. The orbit of any point of $M$ under this one parameter subgroup intersects the axes of $f$, hence the curvature of $M$ is constant.\\
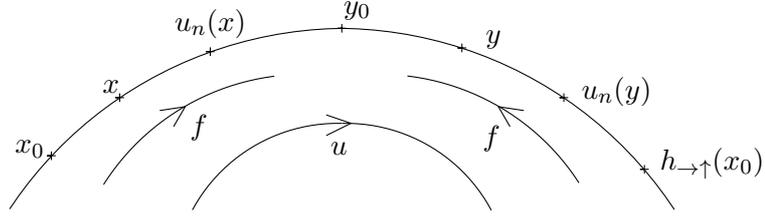
\begin{figure}[h] 
\begin{tikzpicture}[line cap=round,line join=round,>=triangle 45,x=1.0cm,y=1.0cm]
\clip(-4.8,1) rectangle (6.5,4.85);
\draw [shift={(0.95,-1.45)}] plot[domain=1.58:2.56,variable=\t]({1*5.31*cos(\t r)+0*5.31*sin(\t r)},{0*5.31*cos(\t r)+1*5.31*sin(\t r)});
\draw [shift={(0.39,0.13)}] plot[domain=1.7:2.58,variable=\t]({1*3.12*cos(\t r)+0*3.12*sin(\t r)},{0*3.12*cos(\t r)+1*3.12*sin(\t r)});
\draw [shift={(0.83,0.43)}] plot[domain=1.55:2.65,variable=\t]({1*2.17*cos(\t r)+0*2.17*sin(\t r)},{0*2.17*cos(\t r)+1*2.17*sin(\t r)});
\draw [shift={(0.93,0.43)}] plot[domain=0.49:1.59,variable=\t]({1*2.17*cos(\t r)+0*2.17*sin(\t r)},{0*2.17*cos(\t r)+1*2.17*sin(\t r)});
\draw [shift={(1.37,0.13)}] plot[domain=0.57:1.45,variable=\t]({1*3.12*cos(\t r)+0*3.12*sin(\t r)},{0*3.12*cos(\t r)+1*3.12*sin(\t r)});
\draw [shift={(0.81,-1.45)}] plot[domain=0.58:1.56,variable=\t]({1*5.31*cos(\t r)+0*5.31*sin(\t r)},{0*5.31*cos(\t r)+1*5.31*sin(\t r)});
\draw (-1.51,2.77)-- (-1.18,2.82);
\draw (-1.33,2.56)-- (-1.18,2.82);
\draw (3.27,2.77)-- (2.94,2.82);
\draw (3.09,2.56)-- (2.94,2.82);
\draw (0.68,2.71)-- (1,2.6);
\draw (1,2.6)-- (0.68,2.44);
\draw (-3.55,2.51) node[anchor=north west] {$x_0$};
\draw (-2.4,3.3) node[anchor=north west] {$x$};
\draw (-1.45,4.25) node[anchor=north west] {$u_n(x)$};
\draw (0.77,4.35) node[anchor=north west] {$y_0$};
\draw (2.64,3.97) node[anchor=north west] {$y$};
\draw (3.9,3.35) node[anchor=north west] {$u_n(y)$};
\draw (4.93,2.41) node[anchor=north west] {$h_{\rightarrow \uparrow}(x_0)$};
\draw (-1.25,2.85) node[anchor=north west] {$f$};
\draw (2.6,2.7) node[anchor=north west] {$f$};
\draw (0.6,2.49) node[anchor=north west] {$u$};
\begin{scriptsize}
\draw [color=black] (0.88,3.86)-- ++(-1.5pt,0 pt) -- ++(3.0pt,0 pt) ++(-1.5pt,-1.5pt) -- ++(0 pt,3.0pt);
\draw [color=black] (-2.04,2.94)-- ++(-1.5pt,0 pt) -- ++(3.0pt,0 pt) ++(-1.5pt,-1.5pt) -- ++(0 pt,3.0pt);
\draw [color=black] (3.8,2.94)-- ++(-1.5pt,0 pt) -- ++(3.0pt,0 pt) ++(-1.5pt,-1.5pt) -- ++(0 pt,3.0pt);
\draw [color=black] (-2.94,2.17)-- ++(-1.5pt,0 pt) -- ++(3.0pt,0 pt) ++(-1.5pt,-1.5pt) -- ++(0 pt,3.0pt);
\draw [color=black] (-0.85,3.55)-- ++(-1.5pt,0 pt) -- ++(3.0pt,0 pt) ++(-1.5pt,-1.5pt) -- ++(0 pt,3.0pt);
\draw [color=black] (2.46,3.6)-- ++(-1.5pt,0 pt) -- ++(3.0pt,0 pt) ++(-1.5pt,-1.5pt) -- ++(0 pt,3.0pt);
\draw [color=black] (4.86,1.99)-- ++(-1.5pt,0 pt) -- ++(3.0pt,0 pt) ++(-1.5pt,-1.5pt) -- ++(0 pt,3.0pt);
\end{scriptsize}
\end{tikzpicture}
\caption{Dynamics in the parabolic case} \label{fig:parabolic}
\end{figure}
\indent Consider the developing map $D:\tilde M\to N$ where $N$ is either $\R^{1,1}$ or $\tilde{dS}_2$. Notice that $D$ is the conformal embedding $\tilde p$ defined in \ref{subsec:flat_cylinder}, hence it is injective and $M$ is a quotient of an open set of $\R^{1,1}$ or $\tilde{dS}_2$. Since it has an isometry group acting non properly, it is an open set of $\tilde{dS}_2$ and $\rho_1^M(\Isom(M,g))$ is differentially conjugate to a quotient of $\tilde{\PSL}(2,R)$. The fact that it has isometries with $k$ fixed points on the circle implies that $\rho_1^M(\Isom(M,g))$ is differentially conjugate to a subgroup of $\PSL_k(2,\R)$.\\
\emph{Second case:} Assume that $u$ is hyperbolic. The commutator $[u,f]\in H$ satisfies $[u,f](x_0)=x_0$ and $[u,f]'(x_0)=1$, therefore it is either parabolic, either the identity. Since $u$ and $f$ have different fixed points, we have $[u,f](y_0)\ne y_0$, and $H$ possesses a parabolic element. \\ 

\textbf{Second step:} $H$ has only parabolic elements.\\
Since $H$ preserves the affine structure on $\{x_0\} \times \intoo{h_{\downarrow}(x_0)}{h_{\uparrow}(x_0)}$ defined by the parametrisation of geodesics, we see that $H$ is isomorphic to a subgroup of the affine group $\textrm{Aff}(\R)$. Parabolic elements are sent to translations in $\textrm{Aff}(\R)$, therefore $H$ is isomorphic to a subgroup of $\R$, and it is either $\Z$ either $\R$. By using exactly the same methods as in the hyperbolic case (second to fifth steps), we see first that $H$, then $G$ are topologically conjugate to subgroups of $\PSL_k(2,\R)$.
\end{proof}

As a byproduct of the proof, we obtain the following result that will simplify our search for counter examples when we study the differential conjugacy problem.

\begin{coro} Let $(M,g)$ be a spatially compact surface with an acausal conformal boundary that embeds in $\T^2$. Let $G\subset \Isom(M,g)$ be an elementary subgroup (i.e. $\rho_1^M(G)$ is elementary) that acts non properly on $M$ . If $\rho_1^M(G)$ is not differentially conjugate to a subgroup of some $\PSL_k(2,\R)$, then there is a finite index subgroup of $G$ that is isomorphic to $\Z$ or $\R$. \end{coro}

\section{Conjugacy for non elementary groups} \label{sec:non_elementary_acausal}

We now wish to show Theorem  \ref{main_theorem_acausal} in the non elementary case. We will deal only with  surfaces that embed conformally in $\T^2$, and conclude with Proposition \ref{open_set_embeds}. The main tool will be the $(h,k)$-convergence property, but we face the problem  that $h$ is not necessarily a rotation. If it were the case, then the result would be immediate by Lemma \ref{convergence_rotation}.\\
\indent Our goal is to see that there is a homeomorphism $\widetilde h$ that is topologically conjugate to a rotation, such that $\rho_1^M(\Isom(M,g))$ is a $(\widetilde h,k)$-convergence group. Our strategy is to keep $h_{\rightarrow \uparrow}$ on the limit set, and to change it outside the limit set in order to make all points periodic. First, we will see that $h_{\rightarrow \uparrow}$ is of order $k$ on the limit set.

\begin{lemme} Let $k\in \N^*$ and let $h\in \Homeo(\Ss^1)$ have rotation number $ \frac{1}{k}$. Let $G\subset \Homeo(\Ss^1)$ be a non elementary $(h,k)$-convergence group. Let $F$ be the set of points $b\in \Ss^1$ such that there is a sequence $f_n\to \infty$ in $G$ and $a\in \Ss^1$ satisfying: \begin{itemize} \item $h^k(a)=a$ and $h^k(b)=b$ \item $\forall i\in\{0,\dots,k-1\}~ \forall x\in \intoo{h^i(a)}{h^{i+1}(a)} ~ f_n(x)\to h^i(b)$ \end{itemize} Then  $L_G\subset \overline F$. In particular, all points of $L_G$ are  periodic for $h$. \end{lemme}

\begin{proof} Since the minimal invariant closed set is unique in the non elementary case, we only need to show that $F$ is $G$-invariant. If $f_n$ is a sequence in $G$ defining $b\in F$, then the sequence $g\circ f_n$ shows that $g(b)\in F$. \end{proof}



\begin{lemme} \label{changing_homeomorphism} Let $(M,g)$ be a spatially compact surface with an acausal conformal boundary that embeds in $\T^2$ and assume that  $G=\rho_1^M(\Isom(M,g))$ is a non elementary subgroup. Let $h_{\uparrow},h_{\downarrow}$ be the homeomorphisms that define the boundary in $\T^2$, and assume that the rotation number of $h_{\rightarrow \uparrow}$ is $\frac{1}{k}$. There is $h\in \Homeo(\Ss^1)$ such that: \begin{itemize} \item $ h^k=Id$ \item $ h(x)=h_{\rightarrow \uparrow}(x)$ for all $x\in L_G$ \item $ h \circ f = f \circ  h$ for all $f\in G$ \end{itemize} Consequently, $G$ is a $( h,k)$-convergence group. \end{lemme}

\begin{proof} If $L_G=\Ss^1$, then $h_{\rightarrow \uparrow}^k=Id$ and we can take $ h=h_{\rightarrow \uparrow}$. We now assume that $L_G$ is a Cantor set. \\ \indent If $I\subset \Ss^1$ is a connected component of $\Ss^1\setminus L_G$, then so is $h_{\rightarrow \uparrow}(I)$, and $h_{\rightarrow \uparrow}^k(I)=I$ (because $h_{\rightarrow \uparrow}^k$ is the identity on $L_G$). We use the decomposition into connected components $\Ss^1\setminus L_G = \bigcup_{n\in \N} (I_n \sqcup h_{\rightarrow \uparrow}(I_n) \sqcup\cdots \sqcup h_{\rightarrow \uparrow}^{k-1}(I_n))$. Let $E\subset \N$ be a set of representers for the action of $G$.\\
\indent Let $n\in E$. Let $G_n$ be the set of elements of $G$ that preserve the union  $I_n\sqcup \cdots \sqcup h_{\rightarrow \uparrow}^{k-1}(I_n)$. It is a closed elementary group with an invariant set containing $2k$ elements.  Lemma \ref{hyperbolic} implies that it is topologically conjugate to subgroup $\Gamma_n$ of $\PSL_k(2,\R)$. We write $G_n=\p^{-1}\Gamma_n\p$. Let $R$ be the element of the center of $\PSL_k(2,\R)$ of rotation number $\frac{1}{k}$. Set $ h$ on $h_{\rightarrow \uparrow}^i(I_n)$ to be $\p^{-1} R \p$. Since $R^k=Id$, we find $ h^k=Id$. Since $R$ commutes with $\Gamma_n$, we see that $ h$ commutes with $G_n$.\\
\indent On $f(h_{\rightarrow \uparrow}^i(I_n))$ for $f\in G$, we set $ h_{/f(h_{\rightarrow \uparrow}^i(I_n))} = f \circ  h_{/h_{\rightarrow \uparrow}^i(I_n)} \circ f^{-1}$. Since $ h$ commutes with the elements that preserve $I_n\sqcup \cdots \sqcup h_{\rightarrow \uparrow}^{k-1}(I_n)$, we see that $ h$ commutes with $G$. By construction, it is equal to $h_{\rightarrow \uparrow}$ on $L_G$, and it is a homeomorphism.
\end{proof}

We can now achieve the proof Theorem \ref{main_theorem_acausal_embeds}, which implies Theorem \ref{main_theorem_acausal} because of Proposition \ref{open_set_embeds}.

\begin{proof}[Proof of Theorem \ref{main_theorem_acausal_embeds}] Let $(M,g)$ be  a spatially compact surface with an acausal conformal boundary that embeds conformally in the flat torus.\\
\indent Proposition \ref{faithful} implies that $\rho_1^M$ and $\rho_2^M$ are faithful and topologically conjugate to each other.\\ \indent If $\Isom(M,g)$ acts properly on $M$, then Proposition \ref{main_acausal_embeds_proper} concludes. If $\Isom(M,g)$ acts non properly on $M$, then according to Proposition \ref{rotation_number}, the rotation number of $h_{\rightarrow \uparrow}$ is equal to $\frac{1}{k}$ for some $k\in \N$.\\ \indent  If $\rho_1^M(\Isom(M,g))$ is elementary, then Lemma \ref{finite_invariant_set} assures that we can apply Lemma \ref{elliptic}, \ref{hyperbolic} or \ref{parabolic} to show that $\rho_1^M(\Isom(M,g))$ is topologically conjugate to a subgroup of a finite cover of $\PSL(2,\R)$.\\ \indent If $\rho_1^M(\Isom(M,g))$ is non elementary, then Lemma \ref{changing_homeomorphism} shows that there is $ h\in \Homeo(\Ss^1)$ such that $ h^k=Id$ and $\rho_1^M(\Isom(M,g))$ is a $( h,k)$-convergence group. Since $ h^k=Id$, it is topologically conjugate to a rotation, and Lemma \ref{convergence_rotation} shows that $\rho_1^M(\Isom(M,g))$ is topologically conjugate to a subgroup of $\PSL_k(2,\R)$.
\end{proof}

\subsection*{Acknowledgments} This work corresponds to chapters 1 and 3 of my PhD thesis \cite{these}. I would like to thank my  advisor Abdelghani Zeghib for his help throughout this work, as well as Thierry Barbot for some important remarks on a first version of this paper.

~\\
\footnotesize \textsc{UMPA, \'Ecole Normale Supérieure de Lyon, 46 allée d'Italie, 69364 Lyon Cedex 07}\\
 \emph{E-mail address:}  \verb|daniel.monclair@ens-lyon.fr|

\end{document}